\documentclass[reqno]{amsart}
\usepackage{url}
\usepackage{amssymb}
\usepackage{graphicx}
\usepackage[colorinlistoftodos]{todonotes}
\usepackage{amsmath}
\usepackage{amsthm}

\newtheorem{theorem}{Theorem}[section]
\newtheorem{lemma}[theorem]{Lemma}
\newtheorem{corollary}[theorem]{Corollary}

\theoremstyle{definition}
\newtheorem{definition}[theorem]{Definition}

\theoremstyle{remark}
\newtheorem{remark}[theorem]{Remark}

\numberwithin{equation}{section}
\usepackage{algorithmic}
\usepackage[top=3cm,bottom=2cm,right=2cm,left=2cm]{geometry}

\begin{document}

\title[$k$-Elongated Plane Partitions and Divisibility by 5]{The Localization Method Applied to $k$-Elongated Plane Partitions and Divisibility by 5}


\author{Koustav Banerjee}
\address{}
\curraddr{}
\email{}
\thanks{}

\author{Nicolas Allen Smoot}
\address{}
\curraddr{}
\email{}
\thanks{}



\date{}

\dedicatory{}

\begin{abstract}
The enumeration $d_k(n)$ of $k$-elongated plane partition diamonds has emerged as a generalization of the classical integer partition function $p(n)$.  We have discovered an infinite congruence family for $d_5(n)$ modulo powers of 5.  Classical methods cannot be used to prove this family of congruences.  Indeed, the proof employs the recently developed localization method, and utilizes a striking internal algebraic structure which has not yet been seen in the proof of any congruence family.  We believe that this discovery poses important implications on future work in partition congruences.
\end{abstract}

\maketitle

\section{Introduction}

The theory of partition congruences has developed substantially over the last century, beginning with Ramanujan's revolutionary congruences for the integer partition function $p(n)$.  Prior to the twentieth century, the sequence $(p(n))_{n\ge 0}$ was thought to be pseudorandom with respect to its divisibility properties.  Ramanujan changed this perception almost overnight, with his discovery of three remarkable infinite congruence families for $p(n)$.  We give these families here, with an appropriate adjustment on the powers of 7.
\begin{align}
\text{If } n,\alpha\in\mathbb{Z}_{\ge 1} \text{ and } 24n\equiv 1\pmod{\ell^{\alpha}},&\text{ then } p(n)\equiv 0\pmod{\ell^{\beta}},\label{origconjecR5711}
\end{align} with
\begin{align}
\beta := 
\begin{cases}
\alpha & \text{ if } \ell\in\{5,11\},\\
\left\lfloor \frac{\alpha}{2} \right\rfloor + 1 & \text{ if } \ell=7.
\end{cases}\label{origconjecR5711B}
\end{align}  Other congruence families have been found to exist for a variety of partition functions and related arithmetic sequences which at least superficially resemble those of (\ref{origconjecR5711}).  However, the difficulty of proving such congruence families can vary substantially, with some contemporary families still resisting proof today.

The congruence family discussed in this paper was discovered by the first author, and regards the enumeration $d_5(n)$ of $5$-elongated plane partitions of $n$.  The function is part of a large class of partition functions, $d_k(n)$, which generalize $p(n)$, and whose properties have been closely studied by partition theorists over the last few years.

The congruence family is interesting on its own; however, the proof method is the most important aspect of our work.  The classical techniques used to prove Ramanujan's results (\ref{origconjecR5711}) are insufficient to complete a proof.  Instead, the recently developed localization method is utilized.  The most critical part of the proof involves a surprising internal algebraic structure on the rational polynomials representing the individual cases of the family.  It is this structure and possible analogues for other functions that we believe may lead to a deeper understanding of the theory of partition congruences, and possible resolution of standing conjectures.
\subsection{$k$-Elongated Plane Partitions}
Over the years George Andrews and Peter Paule have produced an extremely influential series of papers developing an algorithmic methodology on MacMahon's Partition Analysis \cite[Vol. II, Section VIII]{MacMahon}.  Among the many contributions of this series of papers has been the development of the $k$-elongated plane partition function $d_k(n)$, which is enumerated by the following generating function:
\begin{align}
D_k(q) := \sum_{n=0}^{\infty}d_k(n)q^n = \frac{(q^2;q^2)_{\infty}^k}{(q;q)_{\infty}^{3k+1}}\label{D2}.
\end{align}  For our purposes, $k\in\mathbb{Z}$ with $k\ge 0$.  In this and all subsequent arguments, we we employ the $q$-Pochhammer notation in which $a,q\in\mathbb{C}$ and $|q|<1$:
\begin{align}
& (a;q)_{0}:=1, \ \ \ (a;q)_{N}:=\prod_{k=0}^{N-1}(1-aq^k), \ n \geq 1;\label{qpochdefna}\\
& (a;q)_{\infty}:=\lim_{N\rightarrow \infty}(a;q)_{N}, \ \ |q|<1.\label{qpochdefnb}
\end{align}

This function counts the number of directed graphs with the form shown in the diagrams below, in which each $a_j\in\mathbb{Z}_{\ge 0}$, a directed edge $a_{b}\rightarrow a_c$ indicates that $a_b\ge a_c$, and $\sum_{j=0}^{m}a_{(2k+1)j+1}=n$.


\begin{figure}[h]
\caption{A typical length 1 $k$-elongated partition diamond.}\label{kelong1}
\centering
\includegraphics{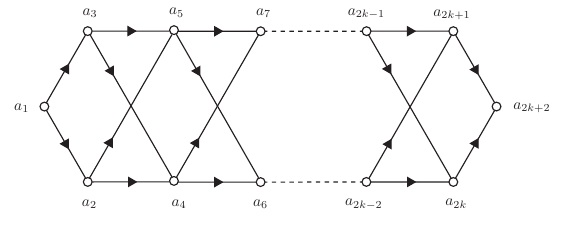}
\end{figure}
\begin{figure}[h]
\caption{A typical length $m$ $k$-elongated partition diamond.}\label{kelong2}
\centering
\includegraphics{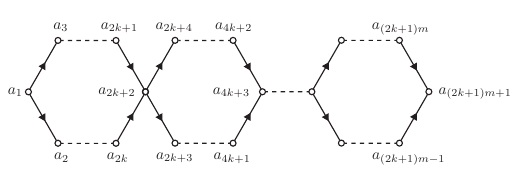}
\end{figure}



Notice that $d_k(n)$ serves as a generalization of the unrestricted integer partition function $p(n)=d_0(n)$.  As such, it may be expected that $d_k(n)$ may exhibit many arithmetic properties similar to those of $p(n)$.  For example, a famous result for $p(n)$ due to Ramanujan \cite{Ramanujan} is
\begin{align}
p(5n+4)\equiv 0\pmod{5}\label{RamOrig5np4}
\end{align} for all $n\ge 0$.  Recently da Silva, Hirschhorn, and Sellers \cite{dasilvaet} discovered that
\begin{align}
d_5(5n+4)\equiv 0\pmod{5}\label{dahirselcong}
\end{align} for all $n\ge 0$.  It is certainly tempting to suspect that \textit{infinite} families of congruence might exist for $d_k(n)$, in form similar to that of Ramanujan's family of congruences for $p(n)$ modulo powers of 5.  These are generally much more difficult to prove than a congruence in a simple arithmetic progression, e.g., (\ref{RamOrig5np4}) or (\ref{dahirselcong}).  Some have already been discovered and proved, e.g., Conjecture 3 of \cite{AndrewsPaule}, which was proved in \cite{Smoot2}.

Our central result is another such family---a simple and elegant extension of (\ref{dahirselcong}) to an infinite family of congruences of $d_5(n)$:
\begin{theorem}\label{Thm12}
Let $n,\alpha\in\mathbb{Z}_{\ge 1}$ such that $4n\equiv 1\pmod{5^{\alpha}}$.  Then $d_5(n)\equiv 0\pmod{5^{\left\lfloor\alpha/2\right\rfloor+1}}$.
\end{theorem}

It is this result that we will prove in this paper.
\subsection{History}\label{historysectionetc}
Much of our modern day understanding of the arithmetic properties of the integer partition function $p(n)$ derives from Ramanujan's groundbreaking work of (\ref{origconjecR5711}) in 1919, albeit without the key modification (\ref{origconjecR5711B}).

This incredible result took decades to prove.  Ramanujan himself appears to have understood the proof in the case that $\ell=5$ \cite{BO}.  The first published proof for $\ell=5,7$ appeared by Watson in 1938 \cite{Watson}.  The case for $\ell=11$ was much more resistant, and a proof was not found before Atkin's work in 1967 \cite{Atkin}.

This result, together with the techniques used to prove it, has stimulated a massive amount of work in partition theory.  In particular, congruence families of a similar form have been found for a wide variety of more restrictive partition functions.  Notably, there exists an extraordinary range of results in terms of the difficulty of proof.  Some results are relatively unremarkable, and are proved regulary using the same methods that were used to prove (\ref{origconjecR5711}) for $\ell=5,7$.

On the other hand, others are much more ambitious, and resistant to proof even today.  For example, a congruence family for 2-colored Frobenius partitions modulo powers of 5 was proposed by Sellers in 1994 \cite{Sellers}, but it was not proved until the work of Paule and Radu in 2012 \cite{Paule}.

Until recently, it was assumed that the genus of the underlying modular curve was responsible for the complications.  For a given partition-like function $a(n)$, a typical infinite congruence family modulo powers of a given prime $\ell$ will have the form
\begin{align*}
a(n)\equiv 0\pmod{\ell^{\beta}},
\end{align*} in which $n$ is the inverse of some fixed positive integer modulo $\ell^{\alpha}$, and
\begin{align*}
\beta\rightarrow\infty\text{ as }\alpha\rightarrow\infty.
\end{align*}  Such a family is generally associated with a sequence of modular functions $\mathcal{L}:=\left( L_{\alpha} \right)_{\alpha\ge 1}$, in which each $L_{\alpha}$ effectively enumerates the congruence of $a(n)$ modulo $\ell^{\alpha}$.  This sequence of modular functions is in turn associated with a given compact Riemann surface $\mathrm{X}$, a classical modular curve.  The topological properties of this curve have a profound impact on the difficulty in proving the associated congruence family.

In particular, it was found that for (\ref{origconjecR5711}) with $\ell=5,7$, the associated modular curves have genus 0.  On the other hand, in the case of (\ref{origconjecR5711}) for $\ell=11$, and for the Andrews--Sellers congruences proved in \cite{Paule}, the associated modular curves have genus 1.  It was therefore assumed that our understanding of partition congruences when $\mathrm{X}$ has genus 0 was complete.  The discovery that classical techniques are not sufficient to prove all genus 0 congruences is quite recent.

Recent work done in \cite{Smoot2} and \cite{Smoot0} has shown that a second important topological property of $\mathrm{X}$ must be taken into account---namely, the cusp count, i.e., the number of points required to properly compactify $\mathrm{X}$.

When the cusp count is 2, the associated functions $L_{\alpha}$ generally have a pole at one cusp and a zero at the other.  The space of functions on $\mathrm{X}$ with a pole only at a single point is isomorphic to a $\mathbb{C}[x]$ module with rank depending on the genus of $\mathrm{X}$.  As a result, we can usually express $L_{\alpha}$ as a polynomial of a finite number of basis functions which have a pole at a single cusp.  Indeed, when the genus is 0, $L_{\alpha}$ is expressible in terms of a single function, i.e., a Hauptmodul.  This is the classical method used to prove nearly all congruence families in the past.

However, when the cusp count is greater than 2, the functions $L_{\alpha}$ may have poles at multiple places.  Thus, a polynomial expression of $L_{\alpha}$ in terms of the basis functions at a given cusp is generally unlikely.  The solution to this dilemma is to express $L_{\alpha}$ as a \textit{rational} polynomial in the basis functions, in which the denominator is governed by some multiplicative set $\mathcal{S}$.

For example, for $\alpha=1$ we have
\begin{align}
L_{1} &= \frac{(q^5;q^5)^{16}_{\infty}}{(q^{10};q^{10})^5_{\infty}} \sum_{n=0}^{\infty}d_5\left(5n + 4\right)q^{n+2}.\label{L1def}
\end{align}  For a more general definition of $L_{\alpha}$, see Section \ref{proofsetup}.  With the Hauptmodul
\begin{align}
x =& q\frac{(q^2;q^2)_{\infty}(q^{10};q^{10})^3_{\infty}}{(q;q)^3_{\infty}(q^{5};q^{5})_{\infty}}\label{xdefn}
\end{align} we have found that
\begin{align}
L_1 =& \frac{5}{(1+5x)^6}\cdot \bigg( 1141x^2+1368024x^3+406830425x^4+56096987730x^{5}+4584273042265x^{6}\nonumber\\
&+252183481030904x^{7}+10080168638009696x^{8}+61564636888325568\cdot 5^{1}x^{9}\nonumber\\
&+1487346172133137920\cdot 5^{1}x^{10}+5831395866875781120\cdot 5^{2}x^{11}\nonumber\\
&+18905742819857817600\cdot 5^{3}x^{12}+51418061000978841600\cdot 5^{4}x^{13}\nonumber\\
&+593053162941844684800\cdot 5^{4}x^{14}+1170186286551219830784\cdot 5^{5}x^{15}\nonumber\\
&+1987670950529108803584\cdot 5^{6}x^{16}+14600412448070041600000\cdot 5^{6}x^{17}\nonumber\\ &+18610929715750581043200\cdot 5^{7}x^{18}+20622155406031349350400\cdot 5^{8}x^{19}\nonumber\\
&+19875756837339796602880\cdot 5^{9}x^{20}
+83256789114237914972160\cdot 5^{9}x^{21}\nonumber\\ &+60509084676946699223040\cdot 5^{10}x^{22}
+38019095473186783887360\cdot 5^{11}x^{23}\nonumber\\ &+102731578880821716582400\cdot 5^{11}x^{24}
+47409173265052493414400\cdot 5^{12}x^{25}\nonumber\\
&+18500160767512490803200\cdot 5^{13}x^{26}
+6023180618118316687360\cdot 5^{14}x^{27}\nonumber\\ &+8030854479703808409600\cdot 5^{14}x^{28}
+1708690327684828364800\cdot 5^{15}x^{29}\nonumber\\ &+278969849417931161600\cdot 5^{16}x^{30}
+164099911422312448000\cdot 5^{16}x^{31}\nonumber\\
&+12384898975268864000\cdot 5^{17}x^{32}
+450359962737049600\cdot 5^{18}x^{33} \bigg).\label{L1defninx}
\end{align}  We see that divisibility by 5 is immediate.  Similar identities can be computed for any $L_{\alpha}$, albeit with a rapidly increasing polynomial degree and mean coefficient size.  We now give our main theorem, from which Theorem \ref{Thm12} immediately follows:
\begin{theorem}\label{Mythm}
Let \begin{align*}
\psi:=\psi(\alpha) &= \left\lfloor \frac{5^{\alpha+1}}{4} \right\rfloor + 1 - \mathrm{gcd}(\alpha,2),\\
\beta:=\beta(\alpha)&=\left\lfloor\alpha/2 \right\rfloor+1.
\end{align*}~Then for all $\alpha\ge 1$, we have \begin{align}
\frac{(1+5x)^{\psi}}{5^{\beta}}&\cdot L_{\alpha}\in\mathbb{Z}[x].\label{myeqna}
\end{align}
\end{theorem}

This theorem is proved by induction.  What has surprised us is that in order to properly complete the induction, we had to take advantage of a striking internal structure attached to the coefficients of $x^n$ in the numerator of the expressions for $L_{\alpha}$.  This structure is defined by the fact that each $L_{\alpha}/5^{\left\lfloor\alpha/2\right\rfloor+1}$ is effectively a member of the kernel of a homomorphism $\Omega$ onto a finite-dimensional $\mathbb{F}_5$-vector space (see Section \ref{algebrastructure}).  In retrospect, this was true in the previous applications of localization in \cite{Smoot2} and \cite{Smoot0}.  Indeed, we recognize this property in all classical cases of partition congruences, including those of $p(n)$.  However, in the cases of \cite{Smoot2} and \cite{Smoot0}, the homomorphism was very simple, and it was only interpreted at the time as a minor idiosyncrasy on the coefficients of the corresponding $L_{\alpha}$.  In the classical cases, e.g., those of $p(n)$, the homomorphism is completely trivial.

We now hypothesize that this kernel structure plays a critical role in the theory of congruence families.  Certainly, it allows a proof to be completed in a relatively straightforward manner, and we are extremely interested to know if this algebraic structure persists---or fails---in other examples.
\subsection{Additional Congruence Results}\label{additionalresultsgiven}
Theorem \ref{Thm12} constitutes our central result, and we believe that our methods used to prove it will prove extremely useful in future work.  However, in searching for Theorem \ref{Thm12} and any other congruence families of similar form, we have also discovered many other interesting congruence results by various powers of 5.  

As such, we include an examination of divisibility properties of $d_k(n)$ by powers of 5 for various values of $k$ and $n$.  These results subsume a significant amount of work done in \cite{dasilvaet}, \cite{Baruahet}.  We have since proved these results using Ramanujan's theta functions and their dissections, and we provide them here for some additional context on the behavior of $d_k(n)$ with respect to progressions with powers of 5.
\begin{theorem}\label{2A} For $k \geq 0$,
	\begin{equation}\label{2Aeqn1}
	d_{25k+1}(25n+23) \equiv 0\pmod{5},
	\end{equation}
	and
	\begin{equation}\label{2Aeqn2}
	d_{75k+16}(25n+8) \equiv 0\pmod{5}.
	\end{equation}
\end{theorem}  
For $k=0$ in (\ref{2Aeqn1}), we have the following result.
\begin{corollary}\cite[Eqn. (6.1), Thm. 6.1]{Baruahet}\label{elongcongcor1}
$$d_{1}(25n+23) \equiv 0\pmod{5}.$$	
\end{corollary}
The following theorem provides an elegant intricate connection between the two parameters; the index of the function $k$, and the initial of the progression $\ell$ considered in the context:
\begin{theorem}\label{1}
	If $k \in \{1,3,5,8,10,13,15,16,18,20,23\}$, then for any non-negative integer $n$:
	\begin{equation}\label{eq1}
	d_k(25n+\ell) \equiv 0\pmod{5};
	\end{equation}
	and for $k \in \{5,8,10\}$,
	\begin{equation}\label{eq2}
	d_k(25n+\ell) \equiv 0\pmod{25},
	\end{equation}
	where $k+\ell=24$.
\end{theorem}
\begin{remark}\label{Newremark}
In Theorem \ref{1}, we have seen that the ``index" parameter $k$ and the ``orbit" parameter $\ell$ are connected by the relation $k+\ell=24$. But this type of phenomena was first observed by Ahmed, Baruah, and Dastidar \cite[Theorem 1.1]{ABD} in their work on congruences for $2$-color partitions. But yet, as it seems, there is no explanation behind why the sum of such two parameter is always $24$ which might be a worthwhile problem to work out for.  
\end{remark}
Finally, we give an interesting infinite family, in which $k$ varies by a linear progression (with $5^3$ as the base), and in which $n$ varies modulo arbitrarily high odd powers of 5, but in which the congruence itself is fixed modulo 5.  
\begin{theorem}\label{3}
	For $n, k \geq 0$, $\alpha \geq 1$ and $j \in \{1,2\}$,
	\begin{equation}\label{3eqn1}
	d_{125k+2}\Bigl(5^{2\alpha+1}n+5^{2\alpha}j+\dfrac{23\times 5^{2\alpha}+1}{8}\Bigr) \equiv 0\pmod{5}.
	\end{equation}
	Precisely,
	\begin{equation}\label{3eqn2}
	d_{125k+2}\Bigl(5^{2\alpha+1}n-\dfrac{1}{8}(5^{2\alpha}-1)+4\times 5^{2\alpha}\Bigr)\equiv d_{125k+2}\Bigl(5^{2\alpha+1}n-\dfrac{1}{8}(5^{2\alpha}-1)+5\times 5^{2\alpha}\Bigr)\equiv 0\pmod{5}.
	\end{equation}
\end{theorem}
Plugging $k=0$ and $\alpha=1$ into (\ref{3eqn2}), we get:
\begin{corollary}\cite[Eqn. (6.3), Thm. 6.1]{Baruahet}
	For $k=0$ and $\alpha=1$,
	\begin{equation*}
	\begin{split}
	d_2(125n+97) &\equiv 0\pmod{5},\\
	d_2(125n+122) &\equiv 0\pmod{5}.\\
	\end{split}
	\end{equation*}
\end{corollary}
\subsection{Outline}
The remainder of this paper is organized as follows: In Section \ref{basictheorysection} we review the properties of Ramanujan's theta functions, together with the theory of modular functions.  In Section \ref{pfthm} we prove Theorems \ref{2A}-\ref{3}.

Thereafter, we commit to proving Theorem \ref{Thm12}.  In Section \ref{proofsetup} we more precisely define the functions $L_{\alpha}$, together with the operators $U^{(\alpha)}$ which map each $L_{\alpha}$ to $L_{\alpha+1}$.  We also give some of the key properties of the Hauptmodul $x$, including the modular equations (\ref{modX}), (\ref{modZ}) which will allow us to build useful recurrence relations for $U^{(\alpha)}\left( x^m/(1+5x)^n \right)$.  In Section \ref{algebrastructure} we construct the function spaces $\mathcal{V}^{(\alpha)}_n$ in which our functions $L_{\alpha}$ live.  In particular, we define the critical homomorphism $\Omega$ whose kernel is closely related to our function spaces.  We then give Theorem \ref{thmuyox}, which demonstrate how the operator $U^{(\alpha)}$ affects elements of $\mathcal{V}^{(\alpha)}_n$.  We also give some important congruence properties of an associated auxiliary function.  In Section \ref{maintheoremsection} we build the induction necessary to finish the proof of Theorem \ref{Mythm}, and with it Theorem \ref{Thm12}.  In Section \ref{initialrelationssection} we show how to finish the proof of Theorem \ref{thmuyox}, together with the proof of (\ref{modX}), (\ref{modZ}), using the tools of the modular cusp analysis.  In Section \ref{finalsection} we discuss the implications of our approach to future work.  In the Appendix we give the ten initial relations of $U^{(\alpha)}(x^m)$ needed to complete the proof of Theorem \ref{thmuyox}.

Throughout Sections \ref{proofsetup}-\ref{initialrelationssection} we devote many of the more tedious computations to a self-contained Mathematica supplement which can be found online at \url{https://www.risc.jku.at/people/nsmoot/d5congsuppA.nb}.
\section{Basic Theory}\label{basictheorysection}
We will recall a few preliminary definitions and lemmas in context of Ramanujan's theta fucntions necessary to prove Theorems \ref{2A}-\ref{3}.  Thereafter, we will give some background in the theory of modular functions required to prove Theorem \ref{Mythm}.
\subsection{Ramanujan's theta functions and its allied results}\hspace{0 cm}
Recall (\ref{qpochdefna})-(\ref{qpochdefnb}), again with $a,q\in\mathbb{C}$ and $|q|<1$.  Ramanujan’s two-variable general theta function is defined as
\begin{equation}\label{Ramtheta}
f_R(a,b):= \sum_{n=-\infty}^{\infty}a^{n(n+1)/2}b^{n(n-1)/2}, \ \  |ab|<1.
\end{equation}
Three special cases of (\ref{Ramtheta}) are defined by, in Ramanujan's notation
\begin{equation*}
\begin{split}
\phi(q)&:=f_R(q,q)=\sum_{n=-\infty}^{\infty}q^{n^2}=(-q;q^2)^2_\infty (q^2;q^2)_\infty,\\
\psi(q)&:=f_R(q,q^3)=\sum_{n=0}^{\infty}q^{n(n+1)/2}=\dfrac{(q^2;q^2)^2_\infty}{(q;q)_\infty},\\
\xi(-q)&:=f_R(-q,-q^2)=\sum_{n=-\infty}^{\infty}(-1)^n q^{n(3n-1)/2}=(q;q)_{\infty}.
\end{split}
\end{equation*} 
Following Ramanujan's definition (\ref{Ramtheta}), Jacobi's famous triple product identity
\begin{equation*}
\sum_{n=-\infty}^{\infty}q^{n^2}z^n=(-qz;q^2)_{\infty}(-q/z;q^2)_{\infty}(q^2;q^2)_{\infty}, \ \ |q|<1 \ \text{and} \ z\neq 0
\end{equation*}
takes the form
\begin{equation}\label{Ramthetaidnet}
f_R(a,b)=(-a,ab)_{\infty} (-b,ab)_{\infty} (ab,ab)_{\infty}.
\end{equation}
Recall the Jacobi's identity \cite[Thm. 1.3.9, p. 14]{Berndt} that reads
\begin{equation}\label{Jacobi}
(q;q)^3_\infty=\sum_{n=0}^{\infty}(-1)^n (2n+1) q^{\frac{n(n+1)}{2}}.
\end{equation}
From \cite[Entry 31,eq. (31.1)]{Berndt}, can express $f_R(a,b)$ as the $n$-linear combination of theta functions in the following form
\begin{equation}\label{Ramthetadsct}
f_R(a,b)=\sum_{r=0}^{n-1}a^{r(r+1)/2}b^{r(r-1)/2}f_R(a^{n(n+1)/2+nr}b^{n(n-1)/2+nr},a^{n(n-1)/2-nr}b^{n(n+1)/2-nr}).
\end{equation}
\begin{lemma}\cite[Entry 10(i), 10(v), and 10(iv), p. 262]{Berndt}\label{lemma0}
	\begin{eqnarray}\label{lemma0eqn1}
	\psi(q^{1/5})&=&q^{3/5}\psi(q^5)+f_R(q^2,q^3)+q^{1/5} f_R(q,q^{4})\\\label{lemma0eqn2}
	\psi^2(q)&=& \dfrac{(q^2;q^2)_{\infty}(q^5;q^5)^3_{\infty}}{(q;q)_{\infty}(q^{10};q^{10})_{\infty}}+q \dfrac{(q^{10};q^{10})^4_{\infty}}{(q^5;q^5)^2_{\infty}}\\\label{lemma0eqn3}
	\phi^2(q^5)&=&\phi^2(q)-4qf_R(q,q^9)f_R(q^3,q^7).
	\end{eqnarray}
\end{lemma}
Now we recall the Roger-Ramanujan continued fraction
\begin{equation}\label{RogRama1}
R(q):={\frac{q^{1/5}}{1}}_{ \ +} \ {\frac{q}{1}}_{ \ +} \ {\frac{q^2}{1}}_{ \ +} \ {\frac{q^3}{1}}_{ \ +\dots} = q^{1/5} \frac{f_R(-q,-q^4)}{f_R(-q^2,-q^3)}, \ |q|<1.
\end{equation}
\begin{lemma}\cite[p. 161 and p. 164]{bcb1}\label{lemma1} If $T(q):=\frac{q^{1/5}}{R(q)}=\frac{f_R(-q,-q^4)}{f_R(-q^2,-q^3)}$,
	\begin{equation}\label{dissect1}
	T(q^5)-q-\frac{q^2}{T(q^5)}=\frac{(q;q)_{\infty}}{(q^{25};q^{25})_{\infty}}.
	\end{equation}
\end{lemma}
\begin{lemma}\cite[p. 165, Eqn. (7.4.14)]{bcb1}\label{lemma2}
	\begin{equation}\label{dissect2}
	\begin{split}
	\frac{1}{(q;q)_{\infty}}= \frac{(q^{25};q^{25})^{5}_{\infty}}{(q^5;q^5)^{6}_{\infty}}&\Biggl(T^{4}(q^5)+q T^{3}(q^5)+2q^2 T^{2}(q^5)+3q^3T(q^5)+5q^4-\frac{3q^5}{T(q^5)}\\
	&+\frac{2q^6}{T^{2}(q^5)}-\frac{q^7}{T^{3}(q^5)}+\frac{q^8}{T^{4}(q^5)}\Biggr).\\
	\end{split}
	\end{equation}
\end{lemma}
\begin{lemma}\cite[Lemma 2.1]{Yao}\label{lemma3}
	\begin{equation}\label{lemma3eqn1}
	\sum_{n=0}^{\infty} a(5n+2)q^n \equiv -2 (q;q)^3_{\infty}(q^2;q^2)^3_{\infty}\ (\text{mod}\ 5),
	\end{equation}
	where
	$a(n)$ is the cubic partition function and its generating function is
	\begin{equation*}
	\sum_{n=0}^{\infty} a(n)q^n=\dfrac{1}{(q;q)_{\infty}(q^2;q^2)_{\infty}}.
	\end{equation*}
\end{lemma}
\subsection{Modular Functions and Riemann Surface Structure}
As with the unrestricted partition function $p(n)$, the generating function of $d_k(n)$ gives our first indication that modular functions may be useful in determining its properties.  Henceforth we will denote
\begin{align*}
q := e^{2\pi i\tau},
\end{align*} with $\tau\in\mathbb{H}$.  We also hold the convention that for any $N\in\mathbb{Z}$ with $N\neq 0$, we have
\begin{align*}
q^{1/N} := e^{2\pi i\tau/N}.
\end{align*}

Recall Dedekind's eta function \cite[Chapter 3]{Knopp}:
\begin{align*}
\eta(\tau):= e^{\pi i\tau/12}\prod_{n=1}^{\infty}\left( 1-e^{2\pi i n\tau} \right),
\end{align*} with domain $\tau\in\mathbb{H}$.  This function is a modular form of half-integral weight with an associated multiplier system \cite[Chapter 3]{Knopp}.  In particular, this means that
\begin{align}
\eta\left(\frac{a\tau+b}{c\tau+d}\right) = \left( -i(c\tau+d) \right)^{1/2}\epsilon(a,b,c,d)\eta(\tau),\label{etamodularrelation}
\end{align} in which $a,b,c,d\in\mathbb{Z}$ such that $ad-bc=1$ and $c>0$, and $\epsilon(a,b,c,d)$ is a specific root of unity (the factor $-i$ is introduced to insure the principal branch of the half-integer power).

It can be quickly demonstrated that for any $j$ with $0\le j\le 4$,
\begin{align*}
\sum_{r=0}^{4} \exp\left( -2\pi i j\cdot \frac{\tau+r}{5} \right) D_k\left( \exp\left( \frac{\tau+r}{5} \right) \right) = 5\sum_{n=0}^{\infty}d_k(5n+j)e^{2\pi i n\tau}.
\end{align*}  Notice that, taking (\ref{D2}) and accounting for the definition of $\eta(\tau)$ above, we have
\begin{align*}
D_k(q) = e^{\pi i \tau (k+1)/12}\frac{\eta(2\tau)^k}{\eta(\tau)^{3k+1}},
\end{align*} whence $D_k(q)$ is in effect a quotient of eta functions.  Because of this, the function $\sum_{n=0}^{\infty}d_k(5n+j)q^n$ should hold similar properties to (\ref{etamodularrelation}).  Indeed, we would like to adjust this function so that it follows a cleaner functional equation, i.e., exact symmetry for $\tau\rightarrow (a\tau+b)(c\tau+d)^{-1}$, for a large range of matrices $\left(\begin{smallmatrix}
  a & b \\
  c & d 
 \end{smallmatrix}\right)\in\mathrm{SL}\left( 2,\mathbb{Z} \right)$.

The subset of matrices we need are the congruence subgroups
\begin{align*}
\Gamma_0(N) = \Bigg\{ \begin{pmatrix}
  a & b \\
  c & d 
 \end{pmatrix}\in \mathrm{SL}(2,\mathbb{Z}) : N|c \Bigg\},
\end{align*} for $N\in\mathbb{Z}_{\ge 1}$.  Indeed, $\Gamma_0(N)$ acts on the extended complex plane
\begin{align*}
\hat{\mathbb{H}} := \mathbb{H}\cup\mathbb{Q}\cup\{\infty\}
\end{align*} by
\begin{align*}
\left(\begin{pmatrix}
  a & b \\
  c & d 
 \end{pmatrix},\tau\right)&\longrightarrow \frac{a\tau+b}{c\tau+d}.
\end{align*}  We express this group action by 
\begin{align*}
\gamma\tau := \frac{a\tau+b}{c\tau+d},
\end{align*} for any $\tau\in\hat{\mathbb{H}}$ and any $\gamma=\begin{pmatrix}
  a & b \\
  c & d 
 \end{pmatrix}$.  We define the orbits of this group action as the sets
\begin{align*}
[\tau]_N := \left\{ \gamma\tau: \gamma\in\Gamma_0(N) \right\}.
\end{align*}
\begin{definition}
For $N\in\mathbb{Z}_{\ge 1}$, the classical modular curve of level $N$ is the set of all orbits of $\Gamma_0(N)$ applied to $\hat{\mathbb{H}}$:
\begin{align*}
\mathrm{X}_0(N):=\left\{ [\tau]_N : \tau\in\hat{\mathbb{H}} \right\}.
\end{align*}
\end{definition}
Most importantly, one can construct a complex structure for the modular curves $\mathrm{X}_0(N)$, and in so doing show that the curves are \textit{Riemann surfaces}.  This allows us to study meromorphic functions on $\mathrm{X}_0(N)$ using most of the tools of classical complex analysis.

A description of the necessary complex structure is given in \cite[Chapters 2,3]{Diamond}.  The topology of $\mathrm{X}_0(N)$ includes two important nonnegative integers: the genus, denoted $\mathfrak{g}\left(  \mathrm{X}_0(N) \right)$, and the number of cusps, denoted $\epsilon_{\infty}\left(  \mathrm{X}_0(N) \right)$.  These can be computed according to formulae in \cite[Chapter 3, Sections 3.1, 3.8]{Diamond}.

We note that the group action of $\Gamma_0(N)$ can be restricted to $\hat{\mathbb{Q}}:=\mathbb{Q}\cup\{\infty\}$.  Indeed, for $\tau\in\hat{\mathbb{Q}}$, we must have $[\tau]_N\subseteq\hat{\mathbb{Q}}$.  Clearly such orbits partition $\hat{\mathbb{Q}}$ (in the set-theoretic sense).  It can be proved that only a finite number of such orbits exist for any $\mathrm{X}_0(N)$ \cite[Section 3.8]{Diamond}.
\begin{definition}\label{equalcuspdefn}
For any $N\in\mathbb{Z}_{\ge 1}$, the cusps of $\mathrm{X}_0(N)$ are the orbits of $\Gamma_0(N)$ applied to $\hat{\mathbb{Q}}$.
\end{definition}
To determine whether two elements of $\hat{\mathbb{Q}}$ are in the same cusp, we have the following \cite[Proposition 3.8.3]{Diamond}:
\begin{lemma}\label{cuspEqualLemma}
Two rational elements $a_1/c_1, a_2/c_2$ are in the same cusp of $\mathrm{X}_0(N)$ if and only if there exist some $j,y\in\mathbb{Z}$ with $\mathrm{gcd}(y,N)=1$ and
\begin{align}
y\cdot a_2&\equiv a_1+j\cdot c_1\pmod{N},\label{equivalentcuspsA}\\
c_2&\equiv y\cdot c_1\pmod{N}.\label{equivalentcuspsB}
\end{align}
\end{lemma}
The appeal of $\mathrm{X}_0(N)$ is that we can construct meromorphic functions which are constrained by the surface's manifold structure.  With this in mind, we give the following definition of a modular function:
\begin{definition}\label{DefnModular}
Let $f:\mathbb{H}\longrightarrow\mathbb{C}$ be holomorphic on $\mathbb{H}$.  Then $f$ is a modular function over $\Gamma_0(N)$ if the following properties are satisfied for every $\gamma=\left(\begin{smallmatrix}
  a & b \\
  c & d 
 \end{smallmatrix}\right)\in\mathrm{SL}(2,\mathbb{Z})$:
\begin{enumerate}
\item If $\gamma\in\Gamma_0(N)$, we have $\displaystyle{f\left( \gamma\tau \right) = f(\tau)}.$
\item We have $$\displaystyle{f\left( \gamma\tau \right) = \sum_{n=n_{\gamma}}^{\infty}\alpha_{\gamma}(n)q^{n\gcd(c^2,N)/ N}},$$  with $n_{\gamma}\in\mathbb{Z}$, and $\alpha_{\gamma}(n_{\gamma})\neq 0$.  If $n_{\gamma}\ge 0$, then $f$ is holomorphic at the cusp $[a/c]_N$.  Otherwise, $f$ has a pole of order $n_{\gamma}$, and principal part
 \begin{align}
 \sum_{n=n_{\gamma}}^{-1}\alpha_{\gamma}(n)q^{n\gcd(c^2,N)/ N},\label{princpartmod}
 \end{align} at the cusp $[a/c]_N$.
 \end{enumerate}  We refer to $\mathrm{ord}_{a/c}^{(N)}(f) := n_{\gamma}(f)$ as the order of $f$ at the cusp $[a/c]_N$.
\end{definition}
\begin{definition}
Let $\mathcal{M}\left(\Gamma_0(N)\right)$ be the set of all modular functions over $\Gamma_0(N)$, and $\mathcal{M}^{a/c}\left(\Gamma_0(N)\right)\subset \mathcal{M}\left(\Gamma_0(N)\right)$ to be those modular functions over $\Gamma_0(N)$ with a pole only at the cusp $[a/c]_N$.  These are commutative algebras with 1, and standard addition and multiplication \cite[Section 2.1]{Radu}.
\end{definition}
A modular function $f\in\mathcal{M}\left( \Gamma_0(N) \right)$ induces a meromorphic function
\begin{align*}
\hat{f}&:\mathrm{X}_0(N)\longrightarrow\mathbb{C}\cup\{\infty\}\\
&:[\tau]_N\longrightarrow f(\tau).
\end{align*}

In fact, there is a one-to-one correspondence between the set of meromorphic functions on $\mathrm{X}_0(N)$ with poles only at the cusps and $\mathcal{M}\left(\Gamma_0(N)\right)$, with a matching correspondence between the function orders \cite[Chapter VI, Theorem 4A]{Lehner}.  Moreover, all possible poles for $\hat{f}$ exist as a subset of $[\tau]_N\subseteq\hat{\mathbb{Q}}$, and (\ref{princpartmod}) represents the principal part of $\hat{f}$ in local coordinates near the cusp $[a/c]_N$.

\begin{theorem}\label{riemannsurfacetheorema}
Let $\mathrm{X}$ be a compact Riemann surface, and let $\hat{f}:\mathrm{X}\longrightarrow\mathbb{C}$ be analytic on the entirety of $\mathrm{X}$.  Then $\hat{f}$ is be a constant function.
\end{theorem}

This is the most important underlying theorem in the subject, as the following corollary demonstrates:

\begin{corollary}
For any $N\in\mathbb{Z}_{\ge 1}$, if $f\in\mathcal{M}\left(\Gamma_0(N)\right)$ has nonnegative order at every cusp of $\Gamma_0(N)$, then $f$ is a constant.
\end{corollary}

This is a useful tool for demonstrating equivalence between modular functions.  If we consider $f,g\in\mathcal{M}^{\infty}\left(\Gamma_0(N)\right)$, in which $f$ and $g$ contain matching principal parts at $[\infty]_N$, then $f-g$, and therefore $\hat{f}-\hat{g}$, has no poles at all, making $\hat{f}-\hat{g}$ (and therefore $f-g$), a constant.
\subsection{Construction of Modular Functions}
We have defined the Riemann surfaces that we will work over, and whose topologies play a significant role in our work.  We now wish to construct explicit modular functions, using $\eta$ as a guide.  The following theorem is due to Newman \cite[Theorem 1]{Newman}:

\begin{theorem}\label{Newmanntheorem}
Let $f = \prod_{\delta | N} \eta(\delta\tau)^{r_{\delta}}$, with $\hat{r} = (r_{\delta})_{\delta | N}$ an integer-valued vector, for some $N\in\mathbb{Z}_{\ge 1}$.  Then $f\in\mathcal{M}\left(\Gamma_0(N)\right)$ if and only if the following apply:
\begin{enumerate}
\item $\sum_{\delta | N} r_{\delta} = 0;$
\item $\sum_{\delta | N} \delta r_{\delta} \equiv 0\pmod{24};$
\item $\sum_{\delta | N} \frac{N}{\delta}r_{\delta} \equiv 0\pmod{24};$
\item $\prod_{\delta | N} \delta^{|r_{\delta}|}$ is a perfect square.
\end{enumerate}
\end{theorem}

An additional advantage in the use of $\eta$ is the comparable ease in computing the order of the associated modular functions at each cusp.  We can compute the orders using the following theorem \cite[Theorem 23]{Radu}, generally attributed to Ligozat:

\begin{theorem}\label{Ligozat}
If $f = \prod_{\delta | N} \eta(\delta\tau)^{r_{\delta}}\in\mathcal{M}\left(\Gamma_0(N)\right)$, then the order of $f$ at the cusp $[a/c]_N$ is given by the following:
\begin{align*}
\mathrm{ord}_{a/c}^{(N)}(f) = \frac{N}{24\gcd{(c^2,N)}}\sum_{\delta | N} r_{\delta}\frac{\gcd{(c,\delta)}^2}{\delta}.
\end{align*}
\end{theorem}

Finally, we give a theorem from \cite[Theorem 39]{Radu} which we can use to compute the orders of modular functions with the coefficients taken in a given arithmetic progression.

\begin{theorem}\label{orderboundmodfunc}
Suppose that for some $M\in\mathbb{Z}_{\ge 1}$ and an integer-valued vector $(r_{\delta})_{\delta | M}$ we define the function $a(n)$ by
\begin{align*}
\sum_{n=1}^{\infty} a(n)q^n := \prod_{\delta | M} (q^{\delta};q^{\delta})_{\infty}^{r_{\delta}}
\end{align*}  Moreover, suppose that $m,t,N\in\mathbb{Z}_{\ge 0}$ such that $0\le t< m$, and that
\begin{align*}
f = \prod_{\lambda | N} \eta(\lambda\tau)^{s_{\delta}} \sum_{n=1}^{\infty} a(mn+t)q^n\in\mathcal{M}\left( \Gamma_0(N) \right).
\end{align*}  Then
\begin{align*}
\mathrm{ord}_{a/c}^{(N)}(f) \ge \frac{N}{\mathrm{gcd}(c^2,N)} \left(\min\limits_{0\le l\le m-1} \frac{1}{24}\sum_{\delta | M} r_{\delta} \frac{\mathrm{gcd}(\delta (a+l\cdot c\cdot \mathrm{gcd}(m^2-1,24)),mc)^2}{\delta m} + \frac{1}{24}\sum_{\lambda | N} \frac{s_{\lambda}\mathrm{gcd}(\lambda, c)^2}{\lambda}\right).
\end{align*}
\end{theorem}
\subsection{The $U_{\ell}$ Operator}
The central functions in our work are $L_{\alpha}$, $\alpha\ge 1$.  To relate $L_{\alpha}$ to $L_{\alpha+1}$, we need the $U_{\ell}$ operator, for a prime $\ell$.  Here, supposing that we have a function
\begin{align*}
f(\tau) = \sum_{m\ge M}a(m)e^{2\pi i m\tau},
\end{align*} for $\tau\in\mathbb{H}$, we will use the following notation in changing variables to $q$:
\begin{align*}
\tilde{f}(q) = \sum_{m\ge M}a(m)q^m.
\end{align*}
\begin{definition}
Let $\tilde{f}(q) = \sum_{m\ge M}a(m)q^m$.  Then
\begin{align}
U_{\ell}\left(\tilde{f}(q)\right) := \sum_{\ell m\ge M} a(\ell m)q^m.\label{defU3}
\end{align}
\end{definition}

We give some properties of $U_{\ell}$.  The proofs can be found in \cite[Chapter 10]{Andrews} and \cite[Chapter 8]{Knopp}.

\begin{lemma}\label{impUprop}
Given two functions 
\begin{align*}
\tilde{f}(q) = \sum_{m\ge M}a(m)q^m,\ \tilde{g}(q) = \sum_{m\ge N}b(m)q^m,
\end{align*} any $\alpha\in\mathbb{C}$, a primitive $\ell$-th root of unity $\zeta$, and the convention that $q^{1/\ell}~:=~e^{2\pi i\tau/\ell}$, we have the following:
\begin{enumerate}
\item $U_{\ell}\left(\alpha\cdot \tilde{f}+\tilde{g}\right) = \alpha\cdot U_{\ell}\left(\tilde{f}\right) + U_{\ell}\left(\tilde{g}\right)$;
\item $U_{\ell}\left(\tilde{f}(q^{\ell})\tilde{g}(q)\right) = \tilde{f}(q) U_{\ell}\left(\tilde{g}(q)\right)$;
\item $\ell\cdot U_{\ell}\left(\tilde{f}\right) = \sum_{r=0}^{\ell-1} \tilde{f}\left( \zeta^rq^{1/\ell} \right)$.
\end{enumerate}
\end{lemma}

We also give an important result from \cite[Lemma 7]{AtkinL}, which we will use in Section \ref{initialrelationssection}:

\begin{theorem}\label{uelleffectsmod}
For any $N\in\mathbb{Z}_{\ge 1}$ with $\ell | N$ and $f\in\mathcal{M}\left( \Gamma_0(N) \right)$,
\begin{align*}
U_{\ell}\left(f \right)&\in\mathcal{M}\left( \Gamma_0(N) \right).
\end{align*}  Moreover, if $\ell^2 | N$, then
\begin{align*}
U_{\ell}\left(f \right)&\in\mathcal{M}\left( \Gamma_0(N/\ell) \right).
\end{align*}
\end{theorem}
\section{Proofs of Theorems \ref{2A}-\ref{3}}\label{pfthm}
Throughout this section, for a formal power series, say $S(q):=\sum_{n=0}^{\infty}s(n)q^n$, for $A\in \mathbb{Z}_{\ge 2}$, we define $A$-dissection of $S(q)$ as
$$S(q)=\sum_{m=0}^{A-1}q^m S_{m}(q^A)\ \text{with}\ S_{m}(q):=\sum_{n= 0}^{\infty}s(An+m)q^{n}.$$
\begin{proof}[Proof of Theorem \ref{2A}]
	Recall that
	\begin{eqnarray}\label{thm2Aeqn1}
	\sum_{n=0}^{\infty}d_1(n)q^n=\dfrac{(q^2;q^2)_\infty}{(q;q)^4_\infty}=\frac{(q;q)_{\infty}(q^2;q^2)_{\infty}}{(q;q)^5_{\infty}}&=&\xi(-q)\xi(-q^2)\left(\dfrac{1}{(q^5;q^5)_\infty}+5 F^*(q)\right)\nonumber\\
	&=:&\dfrac{1}{(q^5;q^5)_\infty}\xi(-q)\xi(-q^2)+5 \widehat{F}(q)
	\end{eqnarray}
where in the third step, we used $(q;q)^5_{\infty}\equiv(q^5;q^5)_{\infty}\ \left(\text{mod}\ 5\right)$ by Binomial theorem, written in the following form:
$$\frac{1}{(q;q)^5_{\infty}}=\frac{1}{(q^5;q^5)_{\infty}}+5F^*(q)\ \text{and}\ \widehat{F}(q)=\xi(-q)\xi(-q^2)F^*(q).$$
 From \eqref{dissect1} with $q\mapsto q^2$, we obtain
	\begin{equation}\label{thm2Aenq2}
	\xi(-q^2)= \xi(-q^{50}) \Biggl(T(q^{10})-q^2-\dfrac{q^4}{T(q^{10})}\Biggr).
	\end{equation}
Combining \eqref{dissect1} and \eqref{thm2Aenq2} together, we get the following $5$-dissection of $	\xi(-q)\xi(-q^2)$,
	\begin{equation*}
	\xi(-q)\xi(-q^2)=:\sum_{m=0}^{4}q^m P^{[1]}_{m}(q^5)
	\end{equation*}
with
\begin{align*}
P^{[1]}_{0}(q^5)&=\xi(-q^{25})\xi(-q^{50})\left(T(q^5)T(q^{10})+\frac{q^5}{T(q^{10})}\right),\\
P^{[1]}_{1}(q^5)&=-\xi(-q^{25})\xi(-q^{50})\left(T(q^{10})-\frac{q^5}{T(q^5)T(q^{10})}\right),\\
P^{[1]}_{2}(q^5)&=-\xi(-q^{25})\xi(-q^{50})\left(T(q^5)+\frac{1}{T(q^5)T(q^{10})}\right),\\
P^{[1]}_{3}(q^5)&=\xi(-q^{25})\xi(-q^{50}),\\
P^{[1]}_{4}(q^5)&=-\xi(-q^{25})\xi(-q^{50})\left(\frac{T(q^5)}{T(q^{10})}+\frac{1}{T(q^5)}\right).\\
\end{align*}	
So, in particular,
	\begin{equation}\label{thm2Aeqn3}
	q^3P^{[1]}_{3}(q^5)=q^3 \xi(-q^{25})\xi(-q^{50}).
	\end{equation}
Considering coefficients of $q^{5n+3}$ on both sides of \eqref{thm2Aeqn1}, it readily implies that
\begin{equation*}
\sum_{n=0}^{\infty}d_1(5n+3)q^{5n+3}= \dfrac{1}{(q^5;q^5)_\infty}q^3P^{[1]}_{3}(q^5)+5 q^3 \widehat{F}_3(q^5),
\end{equation*}	
and applying \eqref{thm2Aeqn3}, we get
	\begin{eqnarray}
	& & \sum_{n=0}^{\infty}d_1(5n+3)q^{5n+3}= \dfrac{1}{(q^5;q^5)_\infty}q^3 \xi(-q^{25})\xi(-q^{50})+5 q^3 \widehat{F}_3(q^5).\nonumber
	\end{eqnarray}
Therefore, canceling the factor $q^3$ and then re-scaling $q^5\mapsto q$, we obtain
\begin{equation}\label{thm2Aeqn4}
\sum_{n=0}^{\infty}d_1(5n+3)q^{n} = \dfrac{1}{(q;q)_\infty} \xi(-q^{5})\xi(-q^{10})+5 \widehat{F}_3(q).
\end{equation}
	Now, setting 
	$$\dfrac{1}{(q;q)_\infty}:=\sum_{m=0}^{4}q^m P^{[2]}_{m}(q^5)$$
	and comparing the coefficients of $q^{5n+4}$ from the $5$-dissection of $(q;q)^{-1}_{\infty}$ in \eqref{dissect2} gives
	\begin{equation}\label{thm2Aeqn5}
 q^4 P^{[2]}_{4}(q^5)=5q^4\dfrac{(q^{25};q^{25})^{5}_{\infty}}{(q^5;q^5)^{6}_{\infty}}.
	\end{equation}
Next, comparing coefficients of $q^{5n+4}$ in \eqref{thm2Aeqn4}, it turns out that
\begin{equation}\label{neweqn1}
\sum_{n=0}^{\infty}d_1\left(5(5n+4)+3\right)q^{5n+4}=\sum_{n=0}^{\infty}d_1\left(25n+23\right)q^{5n+4} = q^4P^{[2]}_{4}(q^5) \xi(-q^{5})\xi(-q^{10})+5q^4\widehat{F}_{3,4}(q^5),
\end{equation}
where $\widehat{F}_{3,4}(q^5)$ comes from the following $5$-dissection of $\widehat{F}_3(q)$:
$$\widehat{F}_3(q)=:\sum_{m=0}^{4}q^m \widehat{F}_{3,m}(q^5).$$
Consequently from \eqref{thm2Aeqn5} and \eqref{neweqn1}, it readily implies that
	\begin{equation}\label{thm2Aeqn6}
	d_1(25n+23) \equiv 0\ (\text{mod}\ 5).
	\end{equation}
	Finally, \eqref{2Aeqn1} follows from \eqref{thm2Aeqn6} and the following observation: for $k \geq 0$,
	\begin{eqnarray}
	\sum_{n=0}^{\infty}d_{25k+1}(n)q^n&=&\dfrac{(q^2;q^2)^{25k+1}_\infty}{(q;q)^{75k+4}_\infty}=\Bigl(\sum_{n=0}^{\infty}d_1(n)q^n\Bigr)\dfrac{\left((q^2;q^2)^{5}_\infty\right)^{5k}}{\left((q;q)^{5}_\infty\right)^{15k}}\nonumber\\
	&\equiv& \Bigl(\sum_{n=0}^{\infty}d_1(n)q^n\Bigr)\dfrac{(q^{10};q^{10})^{5k}_\infty}{(q^5;q^5)^{15k}_\infty}\ (\text{mod}\ 5)\nonumber,
	\end{eqnarray}
where in the last line using Binomial theorem, we used the fact that $(q^{\alpha};q^{\alpha})^{5\beta}_{\infty}\equiv (q^{5\alpha};q^{5\alpha})^{\beta}_{\infty}\ (\text{mod}\ 5)$ for $(\alpha,\beta)\in \mathbb{N}^2$. This concludes the proof of \eqref{2Aeqn1}.
	
	Now, for all $k \geq 0$,
	\begin{eqnarray}\label{thm2Aeqn7}
	\sum_{n=0}^{\infty}d_{75k+16}(n)q^n &=& \dfrac{(q^2;q^2)^{75k+16}_\infty}{(q;q)^{225k+49}_\infty}\nonumber\\
	&\equiv& \dfrac{(q^{10};q^{10})^{15k+3}_\infty}{(q^5;q^5)^{45k+10}_\infty}\xi(-q)\xi(-q^2)\ (\text{mod}\ 5).
	\end{eqnarray}
	Using \eqref{thm2Aeqn3}, we see that
	\begin{eqnarray}\label{thm2Aeqn8}
	\sum_{n=0}^{\infty}d_{75k+16}(5n+3)q^{n} &\equiv& \dfrac{(q^2;q^2)^{15k+3}_\infty}{(q;q)^{45k+10}_\infty}\xi(-q^{5})\xi(-q^{10})\ (\text{mod}\ 5)\nonumber\\
	&\equiv&\dfrac{(q^{10};q^{10})^{3k}_\infty}{(q^5;q^5)^{9k+1}_\infty}\xi(-q^{5})\xi(-q^{10}) (q^2;q^2)^3_\infty\ (\text{mod}\ 5).
	\end{eqnarray}
	Applying the substitution $q\mapsto q^2$ into \eqref{Jacobi}, it follows that the coefficients of $q^{5n+1}$ in $(q^2;q^2)^3_\infty$ is divisible by $5$ because $n(n+1)=5m+1$ if and only if $n\equiv 2\ (\text{mod}\ 5)$ and so, $2n+1\equiv 0 \left(\text{mod}\ 5\right)$ and therefore from \eqref{thm2Aeqn8} with shifting $n\mapsto 5n+1$, we conclude that
	\begin{equation*}
	d_{75k+16}(25n+8) \equiv 0\ (\text{mod}\ 5).   
	\end{equation*}
\end{proof}
\begin{proof}[Proof of Theorem \ref{1}]
	For $k \in \{3,5,8,10,13,15,18,20,23\}$, \eqref{eq1} is an immediate consequence of Corollary $14$ and $15$ of \cite{AndrewsPaule} and Corollary $4.5$ of \cite{dasilvaet}. Equation \eqref{eq1} for $k\in \{1,16\}$ is a special instance of \eqref{2Aeqn1} and \eqref{2Aeqn2}. For the proof of \eqref{eq2}, one can use the Mathematica package RaduRK developed by the second author \cite{Smoot1}.
\end{proof}
\begin{proof}[Proof of Theorem \ref{3}]
From the definition of $d_k(n)$ for $k=2$, we have:
\begin{eqnarray}\label{thm3eqn1}
\sum_{n=0}^{\infty}d_2(n)q^n &=& \dfrac{(q^2;q^2)^2_\infty}{(q;q)^7_\infty}=\frac{1}{(q;q)^5_{\infty}}\dfrac{\psi^2(q)}{(q^2;q^2)^2_\infty}\nonumber\\
&\equiv& \dfrac{1}{(q^5;q^5)_\infty}\dfrac{\psi^2(q)}{(q^2;q^2)^2_\infty}\ (\text{mod}\ 5)\nonumber\\
&\equiv& \dfrac{1}{(q^5;q^5)_\infty}\dfrac{1}{(q^2;q^2)^2_\infty}\left(\dfrac{(q^2;q^2)_\infty(q^5;q^5)^3_\infty}{(q;q)_\infty (q^{10};q^{10})_\infty}+q\dfrac{(q^{10};q^{10})_\infty}{(q^5;q^5)^2_\infty}\right)\ (\text{mod}\ 5)\ (\text{by}\ \eqref{lemma0eqn2})\nonumber\\
&\equiv& \left(\dfrac{(q^5;q^5)^2_\infty}{(q^{10};q^{10})_\infty}\dfrac{1}{(q;q)_\infty(q^2;q^2)_\infty}+q \dfrac{(q^{10};q^{10})^3_\infty}{(q^{5};q^{5})^3_\infty}(q^2;q^2)^3_\infty\right) (\text{mod}\ 5)\nonumber\\
&=:&\dfrac{(q^5;q^5)^2_\infty}{(q^{10};q^{10})_\infty}\left(\sum_{n=0}^{\infty}a(n)q^n\right)+q \dfrac{(q^{10};q^{10})^3_\infty}{(q^{5};q^{5})^3_\infty}(q^2;q^2)^3_\infty+5 H(q),
\end{eqnarray}
where, $H(q)$ is a formal power series with integral coefficients\footnote{Explicit information on $H(q)$ is not necessary for the proof}. Applying \eqref{Jacobi} with the substitution $q\mapsto q^2$, we get
\begin{equation*}
J(q):=(q^2;q^2)^3_\infty=\sum_{n=0}^{\infty}(-1)^n (2n+1) q^{n(n+1)}.
\end{equation*}
Next, observe that the coefficients of $q^{5m+1}$ in $(q^2;q^2)^3_\infty$ is divisible by $5$ because $n(n+1)=5m+1$ if and only if $n\equiv 2\ (\text{mod}\ 5)$ and so, $2n+1\equiv 0 \left(\text{mod}\ 5\right)$. Therefore, following the $5$-dissection of $J(q)$, we can write the component of $J(q)$ consisting of powers of $q$ that are congruent to $1(\text{mod}\ 5)$ as
\begin{equation}\label{thm3eqn2}
q J_1(q^5)=5 qJ^*(q^5)\ \text{where}\ J(q)=:\sum_{m=0}^{4}q^m J_m(q^5).
\end{equation}
Comparing coefficients on both sides of $q^{5n+2}$ in \eqref{thm3eqn1} and then using \eqref{thm3eqn2}, it follows that
\begin{equation*}
\sum_{n=0}^{\infty}d_2(5n+2)q^{5n+2}= \dfrac{(q^5;q^5)^2_\infty}{(q^{10};q^{10})_\infty}\left(\sum_{n=0}^{\infty}a(5n+2)q^{5n+2}\right)+5q^2\dfrac{(q^{10};q^{10})^3_\infty}{(q^{5};q^{5})^3_\infty}J^*(q^5)+5q^2H_2(q^5)
\end{equation*}
with $H_2(q^5)$ comes from the $5$-dissection of $H(q)$. Therefore, canceling the factor $q^3$ and then re-scaling $q^5\mapsto q$ gives
\begin{eqnarray}\label{thm3eqn3}
\sum_{n=0}^{\infty}d_2(5n+2)q^{n} &=& \dfrac{(q;q)^2_\infty}{(q^2;q^2)_\infty}\sum_{n=0}^{\infty}a(5n+2)q^n+5\dfrac{(q^{2};q^{2})^3_\infty}{(q;q)^3_\infty}J^*(q)+5H_2(q) \nonumber\\
&=& -2 (q;q)^5_\infty(q^2;q^2)^2_\infty+5\dfrac{(q;q)^2_\infty}{(q^2;q^2)_\infty} G(q)+5\dfrac{(q^{2};q^{2})^3_\infty}{(q;q)^3_\infty}J^*(q)+5H_2(q)\ (\text{by}\ \eqref{lemma3eqn1})\nonumber\\
&=:& -2 (q^5;q^5)_\infty \psi(q) \xi(-q)+5 G^*(q),
\end{eqnarray}
where in the second step we used \eqref{lemma3eqn1} in the following form
$\sum_{n=0}^{\infty}a(5n+2)q^n= -2 (q;q)^3_\infty(q^2;q^2)^3_\infty+5G(q)$, and in the last step we used $(q;q)^5_{\infty}\equiv (q^5;q^5)_{\infty}\ \left(\text{mod}\ 5\right)$ and defined the residual formal $q$-series as $G^*(q)$. 

Now, from \eqref{dissect1}, we write the $5$-dissection of $\xi(-q)$ as 
$$\xi(-q)=:\sum_{m=0}^{4}q^m \xi_m(q^5)$$
with 
\begin{align}\label{thm3eqn4}
\xi_0(q^5)&=(q^{25};q^{25})_{\infty}T(q^5), \xi_{1}(q^5)=-(q^{25};q^{25})_{\infty}, \xi_{2}(q^5)=-\frac{(q^{25};q^{25})_{\infty}}{T(q^5)T(q^{10})}, \xi_{3}(q^5)=\xi_4(q^5)=0.
\end{align}
Similarly, applying the substitution $q\mapsto q^5$ into \eqref{lemma0eqn1} gives
$$\psi(q)=:\sum_{m=0}^{4}q^m \psi_m(q^5)$$
with 
\begin{align}\label{thm3eqn5}
\psi_0(q^5)&=f_R(q^{10};q^{15}), \psi_{1}(q^5)=f_R(q^{5};q^{20}), \psi_{3}(q^5)=\psi(q^{25}), \psi_{2}(q^5)=\xi_4(q^5)=0.
\end{align}
Therefore, setting
$$\psi(q)\xi(-q)=:\sum_{m=0}^{4}q^mP^{[3]}_m(q^5),$$
and applying \eqref{thm3eqn4} and \eqref{thm3eqn5}, we see that
\begin{equation}\label{neweqn2}
P^{[3]}_4(q^5)=\psi_{3}(q^5)\xi_1(q^5)=-\psi(q^{25}) \xi(-q^{25}).
\end{equation}
Extracting the power series with coefficients of $q^{5n+4}$ on both sides of \eqref{thm3eqn3} and then with \eqref{neweqn2}, we have
\begin{equation*}
\sum_{n=0}^{\infty}d_2(5(5n+4)+2)q^{5n+4}=\sum_{n=0}^{\infty}d_2(25n+22)q^{5n+4}
\equiv 2 q^4(q^5;q^5)_\infty \psi(q^{25}) \xi(-q^{25}) \ (\text{mod}\ 5),
\end{equation*}
and therefore,
\begin{equation}\label{thm3eqn6}
\sum_{n=0}^{\infty}d_2(25n+22)q^{n}
\equiv 2 \xi(-q)\psi(q^5) \xi(-q^5) \ (\text{mod}\ 5).
\end{equation}
By induction, one can conclude that for all $\alpha \in \mathbb{Z}_{\geq 1}$,
\begin{equation}\label{thm3eqn7}
\sum_{n=0}^{\infty}d_2\Biggl(5^{2\alpha}n-\dfrac{1}{8}(5^{2\alpha}-1)+5^{2\alpha}\Biggr)q^{n}
\equiv 2 \xi(-q)\psi(q^5) \xi(-q^5) \ (\text{mod}\ 5).
\end{equation}
Considering the power series with coefficients of $q^{5n+t}$ with $t=\{3,4\}$ on both sides of \eqref{thm3eqn7} yields
$$\sum_{n=0}^{\infty}d_2\Biggl(5^{2\alpha}(5n+t)-\dfrac{1}{8}(5^{2\alpha}-1)+5^{2\alpha}\Biggr)q^{5n+t}
\equiv 2 q^t\xi_t(q^5)\psi(q^5) \xi(-q^5)\ (\text{mod}\ 5).$$
Moreover, with \eqref{thm3eqn4}, it immediately gives
\begin{equation}\label{thm3eqn8}
\sum_{n=0}^{\infty}d_2\Biggl(5^{2\alpha}(5n+t)-\dfrac{1}{8}(5^{2\alpha}-1)+5^{2\alpha}\Biggr)q^{5n+t}
\equiv 0\  (\text{mod}\ 5).
\end{equation}
Applying \eqref{thm3eqn8} into \eqref{thm3eqn7} for $t=\{3,4\}$, we get
\begin{equation}\label{thm3eqn9}
d_{2}\Bigl(5^{2\alpha+1}n-\dfrac{1}{8}(5^{2\alpha}-1)+4\times 5^{2\alpha}\Bigr)\equiv d_{2}\Bigl(5^{2\alpha+1}n-\dfrac{1}{8}(5^{2\alpha}-1)+5\times 5^{2\alpha}\Bigr)\equiv 0\ (\text{mod}\ 5).
\end{equation}
Finally, we see that for all $k \geq 0$,
\begin{eqnarray}\label{thm3eqn10}
\sum_{n=0}^{\infty}d_{125k+2}(n)q^n &=& \dfrac{(q^2;q^2)^2_\infty}{(q;q)^7_\infty}\dfrac{(q^2;q^2)^{125k}_\infty}{(q;q)^{375k}_\infty}\nonumber\\
&\equiv& \Bigl(\sum_{n=0}^{\infty}d_2(n)q^n\Bigr)\dfrac{(q^{10};q^{10})^{25k}_\infty}{(q^5;q^5)^{75k}_\infty}\ (\text{mod}\ 5).
\end{eqnarray}
Equations \eqref{thm3eqn9} and \eqref{thm3eqn10} together imply \eqref{3eqn2}.
\end{proof}
\section{Setup for Proof of Theorem \ref{Mythm}}\label{proofsetup}
The remainder of our paper is dedicated to the proof of Theorem \ref{Thm12}.  The initial setup is standard to the theory.  We define the function
\begin{align}
\mathcal{A} &:= q^6\frac{D_5(q)}{D_5(q^{25})} = q^6\frac{(q^2;q^2)^5_{\infty}(q^{25};q^{25})^{16}_{\infty}}{(q;q)^{16}_{\infty}(q^{50};q^{50})^5_{\infty}}.\label{Zdef}
\end{align}  Using this function together with the $U_5$ operator, we can construct a sequence of linear operators:
\begin{align}
U^{(1)}(f) :=& U_5(f),\\
U^{(0)}(f) :=& U_5(\mathcal{A}\cdot f),\\
U^{(\alpha)}(f) :=& U^{(i)}(f),\ \alpha\equiv i\pmod{2}.
\end{align}  We next define our function sequence $\mathcal{L}:=\left( L_{\alpha} \right)_{\alpha\ge 1}$:
\begin{align}
L_0 &:= 1,\label{L0d}\\
L_{2\alpha-1}(\tau) &= \frac{(q^5;q^5)^{16}_{\infty}}{(q^{10};q^{10})^5_{\infty}}\cdot\sum_{n=0}^{\infty} d_5(5^{2\alpha-1}n + \lambda_{2\alpha-1})q^{n+2},\label{Lod}\\
L_{2\alpha}(\tau) &= \frac{(q;q)^{16}_{\infty}}{(q^2;q^2)^5_{\infty}}\cdot\sum_{n=0}^{\infty} d_5(5^{2\alpha}n + \lambda_{2\alpha})q^{n+1}.\label{Led}
\end{align}  We define $\lambda_{\alpha}$ as
\begin{align}
\lambda_{\alpha} &:= \frac{1+5^{\alpha}\cdot 3}{4},\label{lamodef}.
\end{align} and it can easily be shown that $y=\lambda_{\alpha}$ is the minimal positive integral solution to
\begin{align*}
4y\equiv 1\pmod{5^{\alpha}}.
\end{align*}  If we apply $U^{(0)}$ to the identity, we get
\begin{align}
U^{(0)}(1)&= U_5\left( q^6\frac{D_5(q)}{D_2(q^{25})} \right) = \frac{1}{D_2(q^5)}\cdot U_5\left( q^6 D_5(q) \right) = \frac{1}{D_5(q^5)}\cdot U_5\left( \sum_{n\ge 6} d_5(n-6)q^n \right)\label{isittruethough}\\
&=\frac{1}{D_5(q^5)}\cdot \sum_{5n\ge 6}d_5(5n-6)q^n =\frac{1}{D_5(q^5)}\cdot \sum_{n=0}^{\infty}d_5(5n+4)q^{n+2}\label{isittruethougha}\\
&= L_1.\label{isittruethoughb}
\end{align}  More generally, we state the following:
\begin{lemma}
For all $\alpha\ge 0$, we have
\begin{align}
L_{\alpha+1} &= U^{(\alpha)}\left( L_{\alpha} \right).\label{Lete}
\end{align}
\end{lemma}
\begin{proof}
\begin{align*}
U^{(2\alpha-1)}\left( L_{2\alpha-1} \right) &=U_5\left( L_{2\alpha-1} \right)\\ &= U_5\left( \frac{1}{D_5(q^5)} \sum_{n\ge 0} d_5\left(5^{2\alpha-1}n + \lambda_{2\alpha-1}\right)q^{n+2} \right)\\
&= \frac{1}{D_5(q)}\cdot U_5\left( \sum_{n\ge 2}d_5\left(5^{2\alpha-1}(n-2) + \lambda_{2\alpha-1}\right)q^{n} \right)\\
&= \frac{1}{D_5(q)}\cdot \sum_{5n\ge 2}d_5\left(5^{2\alpha-1}(5n-2) + \lambda_{2\alpha-1}\right)q^{n}\\
&= \frac{1}{D_5(q)}\cdot \sum_{n\ge 1}d_5\left(5^{2\alpha}n-2\cdot 5^{2\alpha-1} + \lambda_{2\alpha-1}\right)q^{n}\\
&= \frac{1}{D_5(q)}\cdot \sum_{n\ge 0}d_5\left(5^{2\alpha}n+5^{2\alpha}-2\cdot 5^{2\alpha-1} + \lambda_{2\alpha-1}\right)q^{n+1}\\
&= \frac{1}{D_5(q)}\cdot \sum_{n\ge 0}d_5\left(5^{2\alpha}n+\lambda_{2\alpha}\right)q^{n+1}.
\end{align*}
\begin{align*}
U^{(2\alpha)}\left( L_{2\alpha} \right) &=U_5\left( \mathcal{A}\cdot L_{2\alpha} \right)\\ &= U_5\left( q^6\frac{D_5(q)}{D_2(q^{25})}\frac{1}{D_5(q)} \sum_{n\ge 0} d_5\left(5^{2\alpha}n + \lambda_{2\alpha}\right)q^{n+1} \right)\\
&= \frac{1}{D_5(q^3)}\cdot U_5\left( \sum_{n\ge 7}d_5\left(5^{2\alpha}(n-7) + \lambda_{2\alpha}\right)q^{n+7} \right)\\
&= \frac{1}{D_5(q^3)}\cdot \sum_{5n\ge 7}d_5\left(5^{2\alpha}(5n-7) + \lambda_{2\alpha}\right)q^{n}\\
&= \frac{1}{D_5(q^3)}\cdot \sum_{n\ge 2}d_5\left(5^{2\alpha+1}n-7\cdot 5^{2\alpha} + \lambda_{2\alpha}\right)q^{n}\\
&= \frac{1}{D_5(q^3)}\cdot \sum_{n\ge 0}d_5\left(5^{2\alpha+1}(n+2)-7\cdot 5^{2\alpha} + \lambda_{2\alpha}\right)q^{n+2}\\
&= \frac{1}{D_5(q^3)}\cdot \sum_{n\ge 0}d_5\left(5^{2\alpha+1}n+\lambda_{2\alpha+1}\right)q^{n+2}.
\end{align*}
\end{proof}
\subsection{Our Hauptmodul}\label{TheModularEquations}
We now have our key function sequence $\mathcal{L}$ defined, as well as the means of constructing $L_{\alpha+1}$ from $L_{\alpha}$.  We now need a reference function with which to properly represent each $L_{\alpha}$.  Using Theorem \ref{orderboundmodfunc}, we can compute
\begin{align}
\mathrm{ord}_{\infty}^{(10)}(L_1) &\ge 1,\\
\mathrm{ord}_{1/3}^{(10)}(L_1) &\ge 5,\\
\mathrm{ord}_{1/2}^{(10)}(L_1) &\ge -6,\\
\mathrm{ord}_{0}^{(10)}(L_1) &\ge -27.
\end{align}  We need to compute a function, preferably a Hauptmodul at $[0]_{10}$, which has positive order at $[1/2]_{10}$.  Only one eta quotient exists which has this property, which we denote
\begin{align}
z = z(\tau) :=& \frac{(q^2;q^2)_{\infty}^{5}(q^5;q^5)_{\infty}}{(q;q)_{\infty}^5 (q^{10};q^{10})_{\infty}}\label{zdef}.
\end{align}  The orders of $z$ at the cusps of $\mathrm{X}_0(10)$ can be computed by Theorem \ref{Ligozat}:
\begin{align}
\mathrm{ord}_{\infty}^{(10)}(z) &= 0,\\
\mathrm{ord}_{1/5}^{(10)}(z) &= 0,\\
\mathrm{ord}_{1/2}^{(10)}(z) &= 1,\\
\mathrm{ord}_{0}^{(10)}(z) &= -1.
\end{align}  Therefore, we know that
\begin{align*}
z^6 L_1 \in\mathcal{M}^0\left( \Gamma_0(10) \right).
\end{align*}  Because $z$ is a Hauptmodul, we should be able to express $z^6 L_1$ as a polynomial in $z$.  However, not only are the coefficients of such an expression are rational, but they possess denominators divisible by very large powers of 5.  Indeed, the expression begins
\begin{align}
z^6 L_1 = \frac{1}{5^{12}} + \frac{22}{5^{12}}z + \frac{198}{5^{12}}z^2 +...\label{L1inz}
\end{align}  This is clearly not a useful representation.

However, we can repair this problem with a small adjustment.  Notice that, because $(1-q)^5\equiv 1-q^5\pmod{5}$, we have
\begin{align}
z\equiv 1\pmod{5}.
\end{align}  Indeed, if we take (\ref{L1inz}) and express $z=1+5x$, then we recover (\ref{L1defninx}).
\begin{lemma}
If $x$ is defined as in (\ref{xdefn}), then
\begin{align}
z = 1+5x\label{zequals1plus9x}
\end{align}
\end{lemma}  It is the function $x$ which will be our most useful reference function.  This lemma may be proved using the modular cusp analysis.  However, we give an elementary proof here.
\begin{proof}
Notice that
\begin{equation*}
z=\dfrac{\xi(-q^2)^5\xi(-q^5)}{\xi(-q)^5\xi(-q^{10})}\ \ \ \and \ \ \ 
x=q\dfrac{\xi(-q^2)\xi(-q^{10})^3}{\xi(-q)^3\xi(-q^{5})}.
\end{equation*}  From here we have
\begin{eqnarray}
z&=& \dfrac{\xi(-q^2)\xi(-q^5)}{\xi(-q)^3\xi(-q^{10})} \psi^2(q)\nonumber\\
&=& \dfrac{\xi(-q^2)\xi(-q^5)}{\xi(-q)^3\xi(-q^{10})}\left(\dfrac{\xi(-q^2)\xi(-q^5)^3}{\xi(-q)\xi(-q^{10})}+q\dfrac{\xi(-q^{10})^4}{\xi(-q^5)}\right)\ \ (\text{by}\ \eqref{lemma0eqn2})\nonumber\\
&=& \dfrac{\xi(-q^2)^2}{\xi(-q)^4}\phi^2(-q^5)+x\nonumber\\
&=& \dfrac{\xi(-q^2)^2}{\xi(-q)^4} \left(\phi^2(-q)+4qf_R(-q,-q^9)f_R(-q^3,-q^7)\right)+x\ \ (\text{from}\ \eqref{lemma0eqn3}\ \text{with}\ q\mapsto-q)\nonumber\\
&=&1+4q\dfrac{\xi(-q^2)^2}{\xi(-q)^4}\dfrac{(q;q^2)_\infty \xi(-q^{10})^2}{(q^5;q^{10})_\infty}+x
=1+4x+x = 1+5x.\nonumber
\end{eqnarray}
\end{proof}
Both $z$ and $x$ will prove useful to us.  As such, we will define two modular equations in each function which will become useful later.
\begin{theorem}
Define
\begin{align*}
a_0(\tau) &= -(625x^5+500x^4+150x^3+20x^2+x)\\
a_1(\tau) &=-(15x+305x^2+2325x^3+7875x^4+10000x^5)\\
a_2(\tau) &= -(85x+1750x^2+13125x^3+46500x^4+60000x^5)\\
a_3(\tau) &=-(215x+4475x^2+35000x^3+122000x^4+160000x^5)\\
a_4(\tau) &= -(205x+4300x^2+34000x^3+120000x^4+160000x^5).
\end{align*}  We then have
\begin{align}
x^5+\sum_{j=0}^4 a_j(5\tau) x^j = 0.\label{modX}
\end{align}
\end{theorem}
\begin{proof}
See the final section.
\end{proof}
\begin{theorem}
Define
\begin{align*}
b_0(\tau) &=-z^5\\
b_1(\tau) &=1+5z+5z^2+5z^3+5z^4-16z^5\\
b_2(\tau) &=-4-15z+10z^2+35z^3+60z^4-96z^5\\
b_3(\tau) &=6+15z-35z^2+40z^3+240z^4-256z^5\\
b_4(\tau) &=-4-5z+20z^2-80z^3+320z^4-256z^5\\
b_5(\tau) &=1.
\end{align*}  We then have
\begin{align}
z^3+\sum_{k=0}^4 b_k(5\tau) z^k = 0.\label{modZ}
\end{align}
\end{theorem}
\begin{proof}
Take (\ref{modX}) and change variables by $z=(x-1)/5$, and simplify.
\end{proof}
\section{Algebra Structure}\label{algebrastructure}
We will now define the spaces of modular functions which our sequence $\mathcal{L}$ will live in, and how the operators $U^{(\alpha)}$ will affect them.
\subsection{Localized Ring}
We will express each member of $\mathcal{L}$ as a rational polynomial in which the denominator consists of powers of $1+5x$.  With this in mind, we define the multiplicatively closed set
\begin{align}
\mathcal{S} := \left\{ (1+5x)^n : n\in\mathbb{Z}_{n\ge 0} \right\}.\label{Sdef}
\end{align}  We will be working with subspaces of the localized ring
\begin{align}
\mathbb{Z}[x]_{\mathcal{S}} := \mathcal{S}^{-1}\mathbb{Z}[x].
\end{align}  To define these subspaces, we begin by defining functions which will serve as a useful lower bound on the 5-adic valuation of the coefficients of powers of $x$ for the associated functions.
\begin{align}
\theta_1(m) &:= \begin{cases} 
   0, &  1\le m\le 7, \\
   \left\lfloor \frac{5m-2}{7} \right\rfloor - 5,       & m\ge 8,
  \end{cases}\label{theta1defn}
\end{align}
\begin{align}
\theta_0(m) &:= \begin{cases} 
   0, &  1\le m\le 4, \\
   \left\lfloor \frac{5m-1}{7} \right\rfloor - 2,       & m\ge 5.
  \end{cases}\label{theta0defn}
\end{align}

We now define the critical mapping $\Omega$ which we discussed in Section \ref{historysectionetc}.  Let
\begin{align}
\Omega&:\bigoplus_{m=2}^{\infty}\mathbb{Z}\rightarrow\mathbb{Z}/5\mathbb{Z}^2,\\
&:\mathbf{s}\mapsto \begin{pmatrix}
  1 & 1 & 2 & 1 & 0 & 0 & 0 & 0 & 0 & 0 ... \\
  0 & 0 & 4 & 0 & 1 & 1 & 1 & 0 & 0 & 0 ...
 \end{pmatrix}\cdot\mathbf{s} = \begin{pmatrix}
  s(2) + s(3) + 2s(4) + s(5) \\
  4s(4) + s(6) + s(7) + s(8)
 \end{pmatrix}\label{omegadefinition}
\end{align} with each component taken modulo 5.  Of course, the index $m$ used in the domain of $\Omega$ can begin with any integer we want.  We begin with $m=2$ because it is convenient when we build the associated function spaces defined below.

Given the form of $L_1$ in (\ref{L1defninx}), we expect to represent each $L_{\alpha}$ in terms of $x^m/(1+5x)^n$ for various nonnegative integers $m,n$.  These are useful reference functions over which we have a great deal of control.

More precisely, if we let $s(m)$ be an arbitrary integer-valued function which is \textit{discrete}, i.e., with finite support, then we define the following spaces:
\begin{align}
\mathcal{V}^{(0)}_n:=& \left\{ \frac{1}{(1+5x)^n}\sum_{m\ge 1} s(m)\cdot 5^{\theta_0(m)}\cdot x^m \right\},\\
\hat{\mathcal{V}}_n:=& \left\{ \frac{1}{(1+5x)^n}\sum_{m\ge 2} s(m)\cdot 5^{\theta_1(m)}\cdot x^m \right\},\\
\mathcal{V}^{(1)}_n:=& \left\{ \frac{1}{(1+5x)^n}\sum_{m\ge 2} s(m)\cdot 5^{\theta_1(m)}\cdot x^m : \left( s(m) \right)_{m\ge 2}\in\ker\left( \Omega \right) \right\}.
\end{align}  

It is very reasonable to ask why we invoke such an arbitrary algebraic mapping as $\Omega$.  The reason is that in our subsequent proofs, especially in lines (\ref{kernelmotivationA})-(\ref{kernelmotivationB}) below, we will require the output sums in (\ref{omegadefinition}) to vanish mod 5.  For $\alpha\ge 1$, we will derive our even-indexed $L_{2\alpha}$ by applying the associated $U^{(1)}$ operator to each of the component functions of $L_{2\alpha-1}$.

Ideally, we would expect that applying $U^{(1)}$ to each component would yield a combination of functions, each of which is divisible by an additional power of 5.  In that case, $L_{2\alpha}$ must clearly be divisible by an extra power of 5 (since each component is so divisible).  This is indeed the case in the proofs of classical congruence families, e.g., those of $p(n)$ modulo powers of 5.

However, this condition is too strong for our problem.  Most components will indeed give us the extra power of 5 that we need; however, some anomalous components \textit{will not} give us this power.  We can compensate for this by showing that these anomalous components essentially \textit{cancel out} when they are combined to form $L_{2\alpha}$.  This \textit{cancellation} property is especially delicate, and we will find that it will only succeed provided that the vector formed by the coefficients $s(m)$ is a member of the kernel of $\Omega$.

As an example of such a function, consider $L_1$, which is given in (\ref{L1defninx}).  If we divide out the common factor of 5, then we have 
\begin{align}
\frac{1}{5}\cdot L_1 =\frac{1}{(1+5x)^6}&\big( 1141x^2 + 1368024x^3 + 406830425x^4 + 56096987730x^5 + 4584273042265x^6\\ &+ 252183481030904x^7 + 10080168638009696x^8 +... \big).\label{L1inducstage1}
\end{align}  We have first to check that $5^{\theta_1(m)}$ divides each coefficient of $x^m$.  This can be seen in our representation of the coefficients of $x^m$ in (\ref{L1defninx}) above.  Thus, 
\begin{align}
\frac{1}{5}\cdot L_1\in\hat{\mathcal{V}}_6.\label{L1inducstage1a}
\end{align}  To confirm membership in $\mathcal{V}_6^{(1)}$, we need to check that $(1141,1368024,406830425,56096987730,...)\in\mathrm{ker}\left( \Omega \right)$.

The critical coefficients that we need to examine are those of $x^m$ for $2\le m\le 8$.  But by (\ref{theta1defn}), $\theta_1(m)=0$ for these $m$.  Therefore, we can simply take the coefficients of $x^m$ in (\ref{L1inducstage1}) and confirm that they abide the conditions of $\Omega$:
\begin{align*}
1141+1368024+(2)406830425+56096987730\ &\equiv 0\pmod{5},\\
(4)406830425+4584273042265+252183481030904+10080168638009696\ &\equiv 0\pmod{5}.
\end{align*}  We therefore have kernel membership, and 
\begin{align}
\frac{1}{5}\cdot L_1\in\mathcal{V}^{(1)}_6.\label{L1inducstage1b}
\end{align}
Indeed, we will later prove that every odd-indexed member of $\mathcal{L}$ will live in $\mathcal{V}^{(1)}_n$ for some associated $n$, and similarly for the even-indexed members of $\mathcal{L}$ with respect to $\mathcal{V}^{(0)}_n$.
\subsection{Recurrence Relation}
We know from Lemma \ref{Lete} how $U^{(\alpha)}$ affects members of $\mathcal{L}$.  We now need to understand how $U^{(\alpha)}$ affects members of $\mathbb{Z}[x]_{\mathcal{S}}$.
\begin{lemma}
For all $m,n\in\mathbb{Z}$, and $i\in\{0,1\}$, we have
\begin{align}
U^{(i)}\left( \frac{x^m}{(1+5x)^n} \right) = -\frac{1}{(1+5x)^5}\sum_{j=0}^{4}\sum_{k=1}^{5} a_j(\tau)b_k(\tau)\cdot U^{(i)}\left( \frac{x^{m+j-5}}{(1+5x)^{n-k}} \right).\label{UmodX}
\end{align}
\end{lemma}
\begin{proof}
We take advantage of the modular equation (\ref{modZ}).  We isolate and divide by $b_0(5\tau)$, then multiply by $z^{-n}$ for an arbitrary $n$:
\begin{align}
b_0(5\tau) &= -\sum_{k=1}^5 b_k(5\tau) z^k,\nonumber \\
1 &=-\frac{1}{b_0(5\tau)} \sum_{k=1}^5 b_k(5\tau) z^k,\nonumber\\
z^{-n} &=-\frac{1}{b_0(5\tau)} \sum_{k=1}^5 b_k(5\tau) z^{-(n-k)},\label{xnegn1}
\end{align}  Now we change variables to express $z$ in terms of $x$:
\begin{align}
(1+5x)^{-n} &=-\frac{1}{b_0(5\tau)} \sum_{k=1}^5 b_k(5\tau) (1+5x)^{-(n-k)}.\label{xnegn2}
\end{align}  We multiply both sides by $x^m$ for arbitrary integer $m$:
\begin{align}
\frac{x^m}{(1+5x)^n} &= -\frac{1}{b_0(5\tau)}\sum_{k=1}^{5} b_k(5\tau)\cdot \frac{x^{m}}{(1+5x)^{n-k}}\nonumber\\
&= \frac{1}{(1+5x(5\tau))^5}\sum_{k=1}^{5} b_k(5\tau)\cdot \frac{x^{m}}{(1+5x)^{n-k}}.\label{xnegn3}
\end{align}  We can now take advantage of the modular equation (\ref{modX}):
\begin{align}
\frac{x^m}{(1+5x)^n} &= \frac{1}{b_0(5\tau)}\sum_{k=1}^{5} b_k(5\tau)\cdot\sum_{j=0}^4 a_j(5\tau) \frac{x^{m+j-5}}{(1+5x)^{n-k}}\nonumber\\
&= -\frac{1}{(1+5x(5\tau))^5}\sum_{j=0}^4 \sum_{k=1}^{5}a_j(5\tau) b_k(5\tau)\cdot \frac{x^{m+j-5}}{(1+5x)^{n-k}}.\label{xnegn4}
\end{align}  We multiply by $\mathcal{A}^{1-i}\cdot$, with $i=0,1$:
\begin{align}
\mathcal{A}^{1-i}\cdot\frac{x^m}{(1+5x)^n} &= \frac{1}{b_0(5\tau)}\sum_{k=1}^{5} b_k(5\tau)\cdot\sum_{j=0}^4 a_j(5\tau) \cdot\mathcal{A}^{1-i}\cdot\frac{x^{m+j-5}}{(1+5x)^{n-k}}\nonumber\\
= -\frac{1}{(1+5x(5\tau))^5}&\sum_{j=0}^4\sum_{k=1}^{5}a_j(5\tau) b_k(5\tau)\cdot \mathcal{A}^{1-i}\cdot \frac{x^{m+j-5}}{(1+5x)^{n-k}}.\label{xnegn5}
\end{align}  Remembering that
\begin{align*}
U_5(f(5\tau)\cdot g(\tau)) = f(\tau)\cdot U_5(g(\tau)),
\end{align*} we can take the $U_5$ operator of both sides of (\ref{xnegn5}) and simplify:
\begin{align}
U_5&\left(\mathcal{A}^{1-i}\cdot \frac{x^m}{(1+5x)^n} \right)\nonumber\\ =& -\frac{1}{(1+5x)^5}\sum_{j=0}^{4}\sum_{k=1}^{5} a_j(\tau)b_k(\tau)\cdot U_5\left(\mathcal{A}^{1-i}\cdot \frac{x^{m+j-5}}{(1+5x)^{n-k}} \right).\label{xnegn6}
\end{align}
\end{proof}
\subsection{The Action of $U^{(i)}$ on $\mathbb{Z}[x]_{\mathcal{S}}$}
The previous lemma will be useful in proving a more direct representation of $U^{(i)}\left( x^m/(1+5x)^n \right)$.  As in the cases of $\mathcal{V}^{(i)}_n$, we will define lower bounds on the 5-adic valuations of the numerator coefficients of $x^r$:
\begin{align*}
\pi_0(m,r) &:= \mathrm{max}\left( 0,\left\lfloor \frac{5r-m+2}{7} \right\rfloor - 5 \right),\\
\pi_1(m,r) &:= \left\lfloor \frac{5r-m}{7} \right\rfloor.
\end{align*}  We also define a useful auxiliary function:
\begin{align}
\phi(l) :&=\left\lfloor \frac{5l+13}{7} \right\rfloor.
\end{align}
\begin{definition}
Let $n\ge 1$.  A function $h:\mathbb{Z}^{n}\rightarrow\mathbb{Z}$ is a \textit{discrete array} if for any fixed $(m_1,m_2,...,m_{n-1})\in\mathbb{Z}^{n-1}$, $h(m_1,m_2,...,m_{n-1},m)$ has finite support with respect to $m$.
\end{definition}
\begin{theorem}\label{thmuyox}
There exist discrete arrays $h_1, h_0 : \mathbb{Z}^{3}\rightarrow\mathbb{Z}$ such that\\
\begin{align*}
U^{(1)}&\left( \frac{x^m}{(1+5x)^n} \right) = \frac{1}{(1+5x)^{5n}}\sum_{r\ge \left\lceil m/5\right\rceil} h_1(m,n,r)\cdot 5^{\pi_1(m,r)}\cdot x^r,\\
U^{(0)}&\left( \frac{x^m}{(1+5x)^n} \right) = \frac{1}{(1+5x)^{5n+6}}\sum_{r\ge \left\lceil (m+1)/5\right\rceil} h_0(m,n,r)\cdot 5^{\pi_0(m,r)}\cdot x^r.
\end{align*}
\end{theorem}
\begin{proof}
We prove the the theorem by induction over $m$ and $n$.  Let us suppose that Theorem \ref{thmuyox} is true for all $m\le m_0-1$, $n\le n_0-1$ for some fixed $m_0, n_0\ge 5$.  For convenience, as we work with the operator $U^{(i)}$, we will define $\delta := 1-i$ and $\kappa := 6-6i$.  We then write
\begin{align}
&U^{(i)}\left( \frac{x^{m_0}}{(1+5x)^{n_0}} \right)\nonumber\\ = -&\frac{1}{(1+5x)^5}\sum_{j=0}^{4}\sum_{k=1}^{5} a_j(\tau)b_k(\tau)\cdot U^{(i)}\left( \frac{x^{m_0+j-5}}{(1+5x)^{n_0-k}} \right)\\
= -&\frac{1}{(1+5x)^5}\sum_{j=0}^{4}\sum_{k=1}^{5} \frac{a_j(\tau)b_k(\tau)}{(1+5x)^{5(n_0-k)+\kappa}}\nonumber\\ 
&\times\sum_{r\ge \left\lceil (m_0+j-5+\delta)/5 \right\rceil} h_i(m_0+j-5,n_0-k,r)\cdot 5^{\pi_i(m_0+j-5,r)}\cdot x^r\\
=&\frac{1}{(1+5x)^{5n_0+\kappa}}\sum_{j=0}^{4}\sum_{k=1}^{5} w(j,k)(\tau)\nonumber\\ 
&\times\sum_{r\ge \left\lceil (m_0+j-5+\delta)/5 \right\rceil} h_i(m_0+j-5,n_0-k,r)\cdot 5^{\pi_i(m_0+j-5,r)}\cdot x^r,
\end{align} wherein we define
\begin{align}
w(j,k)(\tau) :=& -a_j(\tau)b_k(\tau)(1+5x)^{5(k-1)}\nonumber\\
=& \sum_{l=1}^{30} v(j,k,l)\cdot 5^{\left\lfloor \frac{5l+j+1}{7} \right\rfloor}\cdot x^l.
\end{align}  Expanding and recombining, we have
\begin{align}
&U^{(i)}\left( \frac{x^{m_0}}{(1+5x)^{n_0}} \right)\nonumber\\ =& \frac{1}{(1+5x)^{5n_0+\kappa}}\sum_{j=0}^{4}\sum_{k=1}^{5}\sum_{l=1}^{30}\sum_{r\ge \left\lceil (m_0+j-5+\delta)/5 \right\rceil}\nonumber\\ &v(j,k,l)\cdot h_i(m_0+j-5,n_0-k,r)\cdot 5^{\pi_i(m_0+j-5,r) + \left\lfloor \frac{5l+j+1}{7} \right\rfloor}\cdot x^{r+l}.\label{mnmodeqnrpli}
\end{align}  We need to verify that the coefficients of $x^{r+l}$ are divisible by $5^{\pi_i(m_0,r+l)}$, and that $r+l$ is bounded below by $\left\lceil (m_0+\delta)/5 \right\rceil$.  On the latter matter, we note that
\begin{align}
r+l &\ge\left\lceil \frac{m_0+j-5+\delta}{5} \right\rceil + l\\ &\ge \left\lceil \frac{m_0+\delta}{5}-\frac{5-j}{5} \right\rceil + l\\
&\ge \left\lceil \frac{m_0+\delta}{5} \right\rceil - 1 + l\\ &\ge \left\lceil \frac{m_0+\delta}{5} \right\rceil.
\end{align}  Regarding the former matter, we first take $i=1$.  Notice that
\begin{align*}
\left\lfloor \frac{M}{7} \right\rfloor + \left\lfloor \frac{N}{7} \right\rfloor \ge \left\lfloor \frac{M+N-6}{7} \right\rfloor.
\end{align*}  From this, we have
\begin{align}
\pi_1(m_0+j-5,r)+\left\lfloor \frac{5l+j+1}{7} \right\rfloor &= \left\lfloor \frac{5r-m_0-j+5}{7} \right\rfloor + \left\lfloor \frac{5l+j+1}{7} \right\rfloor\\
&\ge \left\lfloor \frac{5(r+l)-m_0}{7} \right\rfloor\\
&=\pi_1(m_0,r).
\end{align}  For $i=0$, we have to more carefully expand 
\begin{align}
&U^{(0)}\left( \frac{x^{m_0}}{(1+5x)^{n_0}} \right)\nonumber\\
=&\frac{1}{(1+5x)^{5n_0+6}}\sum_{j=0}^{4}\sum_{k=1}^{5} w(j,k)(\tau)\nonumber\\ 
&\times\Big(\sum_{\left\lceil (m_0+j-5+1)/5 \right\rceil \le r\le \left\lceil (m_0+j-5+33)/5 \right\rceil-1} h_i(m_0+j-5,n_0-k,r)\cdot x^r\\ 
&+\sum_{r\ge \left\lceil (m_0+j-5+33)/5 \right\rceil} h_0(m_0+j-5,n_0-k,r)\cdot 5^{\left\lfloor \frac{5r-(m_0+j-5)+2}{7} \right\rfloor - 5}\cdot x^r\Big).
\end{align}   Expanding $w(j,k)(\tau)$ again,
\begin{align}
&U^{(0)}\left( \frac{x^{m_0}}{(1+5x)^{n_0}} \right)\nonumber\\ =& \frac{1}{(1+5x)^{5n_0+6}}\sum_{j=0}^{4}\sum_{k=1}^{5}\sum_{l=1}^{30}\Big( \sum_{\left\lceil (m_0+j-5+1)/5 \right\rceil \le r\le \left\lceil (m_0+j-5+33)/5 \right\rceil-1}\nonumber\\ &v(j,k,l)\cdot h_0(m_0+j-5,n_0-k,r)\cdot 5^{\left\lfloor \frac{5l+j+1}{7} \right\rfloor}\cdot x^{r+l}\\
&+\sum_{r\ge \left\lceil (m_0+j-5+33)/5 \right\rceil}\nonumber\\ &v(j,k,l)\cdot h_0(m_0+j-5,n_0-k,r)\cdot 5^{\left\lfloor \frac{5r-(m_0+j-5)+2}{7} \right\rfloor - 5 + \left\lfloor \frac{5l+j+1}{7} \right\rfloor}\cdot x^{r+l}\Big).\label{mnmodeqnrpl}
\end{align}
For $r\ge \left\lceil (m_0+j-5+33)/5 \right\rceil$, we have $\pi_0(m_0+j-5,r) = \left\lfloor \frac{5r-(m_0+j-5)+2}{7} \right\rfloor - 5\ge 0$.  We find that the coefficient of $x^{r+l}$ has a 5-adic valuation of at least
\begin{align}
\left(\left\lfloor \frac{5r-(m_0+j-5)+2}{7} \right\rfloor - 5\right) + \left\lfloor \frac{5l+j+1}{7} \right\rfloor \ge \left\lfloor \frac{5(r+l) - m_0 + 2}{7} \right\rfloor - 5 = \pi_0(m_0,r+l).
\end{align}  For $\left\lceil (m_0+j-5+1)/5 \right\rceil \le r\le \left\lceil (m_0+j-5+33)/5 \right\rceil-1$, we have $\pi_0(m_0+j-5,r)=0$.  Let us suppose that $r+l\le \left\lceil (m_0+j-5+33)/5 \right\rceil-1$.  In this case, we only need the power of 5 to be 0.

On the other hand, if $r+l\ge \left\lceil (m_0+j-5+33)/5 \right\rceil$, then we have
\begin{align}
&\left\lfloor \frac{5l+j+1}{7} \right\rfloor - \left(\left\lfloor \frac{5(r+l)-m_0+2}{7} \right\rfloor - 5\right)\\
&=\left\lfloor \frac{5l+j+36}{7} \right\rfloor -\left\lfloor \frac{5(r+l)-m_0+2}{7} \right\rfloor\\
&\ge \left\lfloor \frac{5(r+l)-m_0+2}{7} \right\rfloor + \left\lfloor \frac{m - 5r +34}{7} \right\rfloor - \left\lfloor \frac{5(r+l)-m_0+2}{7} \right\rfloor\\
&=\left\lfloor \frac{m_0 - 5r +34}{7} \right\rfloor.
\end{align}  We need to show that $m_0 - 5r +34\ge 0$.  But of course this is equivalent to
\begin{align}
r\le \frac{m_0+34}{5},
\end{align} which must be true since
\begin{align}
r\le \left\lceil \frac{m_0+j-5+33}{5} \right\rceil-1\le \left\lceil \frac{m_0+32}{5} \right\rceil.
\end{align}  Therefore in all cases for either value of $i$ we have the coefficient of $x^{r+l}$ divisible by $5^{\pi_i(m_0,r+l)}$.
\begin{align}
&U^{(i)}\left( \frac{x^{m_0}}{(1+5x)^{n_0}} \right)\nonumber\\ =& \frac{1}{(1+5x)^{5n_0+\kappa}}\sum_{\substack{0\le j\le 4,\\ 1\le k\le 5,\\ 1\le l\le 30}} \sum_{r\ge \left\lceil \frac{m_0+\delta}{5} \right\rceil}\nonumber\\ &v(j,k,l)\cdot h_i(m_0+j-5,n_0-k,r-l)\cdot 5^{\pi_i(m_0+j-5,r) + \left\lfloor \frac{5l+j+1}{7} \right\rfloor}\cdot x^{r}.
\end{align}  
\begin{align}
H_i(m,n,r) := \begin{cases}
\sum_{\substack{0\le j\le 4,\\ 1\le k\le 5,\\ 1\le l\le 30}} &\sum_{r\ge \left\lceil \frac{m+\delta}{5} \right\rceil - 1 + l} \hat{H}(i,j,k,l,r),\ r\ge l\\
0,\ r<l
\end{cases}\label{Hi}
\end{align}
\begin{align*}
\hat{H}(i,j,k,l,r) := v(j,k,l)\cdot h_i(m+j-5,n-k,r-l)\cdot 5^{\epsilon(i,j,l,m,r)},
\end{align*}
\begin{align*}
\epsilon(i,j,l,m,r) := \pi_i(m+j-5,r-l) + \left\lfloor \frac{5l+j+1}{7} \right\rfloor - \pi_i(m,r)
\end{align*}
\begin{align}
U^{(i)}\left( \frac{x^{m_0}}{(1+5x)^{n_0}} \right) &= \frac{1}{(1+5x)^{5n_0+\kappa}}\nonumber\\ &\times\sum_{r\ge \left\lceil \frac{m_0+\delta}{5} \right\rceil} H_i(m_0,n_0,r)\cdot 5^{\pi_i(m_0,r)}\cdot x^{r}.
\end{align}
\end{proof}
\subsection{An Important Inequality}
We will need a more detailed understanding of the arithmetic properties of the functions $h_i$.  With this in mind, we give the following useful inequality: 
\begin{lemma}\label{piphippibound}
For $l\ge 0$ and $i\in\{0,1\}$, we have
\begin{align}
\pi_i(m,r-l) + \phi(l) - \pi_i(m,r)\ge 1.\label{inexeq}
\end{align}
\end{lemma}
\begin{proof}
Notice that for $l=0$, the inequality reduces to 
\begin{align}
\phi(0)\ge 1,
\end{align} which is true since $\phi(0)=1$.  We therefore consider the case in which $l>0$.

\subsubsection{The Case for $i=0$}
\noindent Consider first that $\pi_0(m,r-l), \pi_0(m,r)=0$.  This reduces to 
\begin{align}
\phi(l)\ge 1,
\end{align} which follows from $\phi(l)\ge \phi(0)\ge 1$.  If $\pi_0(m,r-l)=0,$ and $\pi_0(m,r)>0$, then
\begin{align}
r&\ge \left\lceil\frac{m+33}{5}\right\rceil\text{ and }r-l\le \left\lceil\frac{m+33}{5}\right\rceil - 1 = \left\lceil\frac{m+28}{5}\right\rceil,\\
l&\ge r - \left\lceil\frac{m+28}{5}\right\rceil.
\end{align}  Knowing this, we can rewrite
\begin{align}
\phi(l) &= \left\lfloor \frac{5l+13}{7} \right\rfloor \ge \left\lfloor \frac{5r-m-28+13}{7} \right\rfloor\\
&= \left\lfloor \frac{5r-m-15}{7} \right\rfloor\\
&\ge \left(\left\lfloor \frac{5r-m+2}{7} \right\rfloor-5\right) + 1\\
&= \pi_0(m,r)+1.
\end{align}  If $\pi_0(m,r-l)>0,$ and $\pi_0(m,r)>0$, then
\begin{align}
\pi_0(m,r-l) +& \phi(l)\\
&= \left\lfloor \frac{5(r-l)-m+2}{7} \right\rfloor + \left\lfloor \frac{5l+13}{7} \right\rfloor\\
&\ge \left\lfloor \frac{5r-m+9}{7} \right\rfloor\\
&= \left\lfloor \frac{5r-m+2}{7} \right\rfloor + 1\\
&= \pi_0(m,r) + 1.
\end{align}
\subsubsection{The Case for $i=1$}
\noindent We simply have
\begin{align}
\pi_1(m,r-l) +& \phi(l)\\
&= \left\lfloor \frac{5(r+l)-m}{7} \right\rfloor + \left\lfloor \frac{5l+13}{7} \right\rfloor\\
&\ge \left\lfloor \frac{5r-m+7}{7} \right\rfloor\\
&= \left\lfloor \frac{5r-m}{7} \right\rfloor + 1\\
&= \pi_1(m,r) + 1.
\end{align}
\end{proof}
\subsection{Some Congruence Properties For $h_i$}
We now state a result which will become extremely useful in the proof of our main theorem:
\begin{theorem}\label{congcor1a}
For all $m,n,r\in\mathbb{Z}_{\ge 1}$ and $i\in\{0,1\}$, we have
\begin{align}
h_i(m,n,r)&\equiv h_i(m,n-5,r)\pmod{5}.\label{hmod222}
\end{align}
\end{theorem}
\begin{proof}
In contrast to the proof of Theorem \ref{thmuyox}, let us examine $U^{(i)}\left( x^{m}/(1+5x)^{n} \right)$ while holding $m$ fixed, and varying $n$ through the modular equation of $1+5x$.  In this case we have
\begin{align}
U^{(i)}&\left( \frac{x^{m_0}}{(1+5x)^{n}} \right)\nonumber\\ =& \frac{1}{(1+5x)^5}\sum_{k=1}^{5}b_k(\tau)\cdot U^{(i)}\left( \frac{x^{m_0}}{(1+5x)^{n-k}} \right)\\
=& \frac{1}{(1+5x)^5}\sum_{k=1}^{5} \frac{b_k(\tau)}{(1+5x)^{5(n-k)-\kappa}} \sum_{r\ge 1} h_i(m_0,n-k,r)\cdot 5^{\pi_i(m_0,r)}\cdot x^r\\
=&\frac{1}{(1+5x)^{5n+\kappa}}\sum_{k=1}^{5} \hat{w}(k)(\tau)\sum_{r\ge 1 } h_i(m_0,n-k,r)\cdot 5^{\pi_i(m_0,r)}\cdot x^r,
\end{align} in which we can expand
\begin{align}
\hat{w}(k)(\tau) :&= b_k(\tau)(1+5x)^{5(k-1)}\nonumber\\
&=\begin{cases}
& \displaystyle{\sum_{l=0}^{20} \hat{v}(k,l)\cdot 5^{\phi(l)}\cdot x^l},\ k<5\\
&\displaystyle{1+\sum_{l=1}^{20} \hat{v}(5,l)\cdot 5^{\phi(l)}\cdot x^l},\ k=5.
\end{cases}
\end{align}  We can now express
\begin{align}
U^{(i)}&\left( \frac{x^{m_0}}{(1+5x)^{n}} \right) = \frac{1}{(1+5x)^{5n+\kappa}}\nonumber\\
\times&\Bigg(\sum_{\substack{1\le k\le 4,\\ 0\le l\le 20,\\ r\ge \left\lceil \frac{m_0+\delta}{5} \right\rceil}} \hat{v}(k,l)\cdot h_i(m_0,n-k,r)\cdot 5^{\pi_i(m_0,r) + \phi(l)}\cdot x^{r+l}\label{reliancenm5a3}\\
&+\sum_{\substack{1\le l\le 20,\\ r\ge \left\lceil \frac{m_0+\delta}{5} \right\rceil}} \hat{v}(5,l)\cdot h_i(m_0,n-5,r)\cdot 5^{\pi_i(m_0,r) + \phi(l)}\cdot x^{r+l}\label{reliancenm5a2}\\
&+\sum_{r\ge \left\lceil \frac{m_0+\delta}{5} \right\rceil} h_i(m_0,n-5,r)\cdot 5^{\pi_i(m_0,r)}\cdot x^{r}\Bigg).\label{reliancenm5a}
\end{align}  Relabeling our powers of $x$, we have
\begin{align}
U^{(i)}\left( \frac{x^{m_0}}{(1+5x)^{n}} \right) = \frac{1}{(1+5x)^{5n+\kappa}}
\times&\Bigg(\sum_{\substack{1\le k\le 4,\\ 0\le l\le 20,\\ r\ge l+ \left\lceil \frac{m_0+\delta}{5} \right\rceil}} \hat{v}(k,l)\cdot h_i(m_0,n-k,r-l)\cdot 5^{\pi_i(m_0,r-l) + \phi(l)}\cdot x^{r}\label{reliancenm5aaaa}\\
&+\sum_{\substack{1\le l\le 20,\\ r\ge l+ \left\lceil \frac{m_0+\delta}{5} \right\rceil}}\hat{v}(5,l)\cdot h_i(m_0,n-5,r-l)\cdot 5^{\pi_i(m_0,r-l) + \phi(l)}\cdot x^{r}\label{reliancenm5aaaaa}\\
&+\sum_{r\ge \left\lceil \frac{m_0+\delta}{5} \right\rceil} h_i(m_0,n-5,r)\cdot 5^{\pi_i(m_0,r)}\cdot x^{r}\Bigg)\label{reliancenm5aaaaaa}\\
= \frac{1}{(1+5x)^{5n+\kappa}}&\sum_{r\ge \left\lceil \frac{m_0+\delta}{5} \right\rceil} h_i(m_0,n,r)\cdot 5^{\pi_i(m_0,r)}\cdot x^{r}.
\end{align}  Notice that $\pi_i(m,r-l) + \phi(l)\ge \pi_i(m,r)+1$ by Lemma \ref{piphippibound}.  Therefore, if we divide out $5^{\pi_i(m,r)}$, we must have
\begin{align*}
h_i(m,n,r)&\equiv h_i(m,n-5,r)\pmod{5}.
\end{align*}
\end{proof}
By explicitly computing $h_1(m,n,r)$ for any explicit $m,r$, and 5 consecutive values of $n$, as we have done in our Mathematica supplement, we have the following result:
\begin{corollary}\label{congcor1b}
For all $n\in\mathbb{Z}$, with $n\equiv 1\pmod{5},$ we have the following componentwise matrix congruence:
\begin{align}
&\begin{pmatrix}
  h_1(2,n,1) & h_1(3,n,1) & h_1(4,n,1) & h_1(5,n,1) & 0 & 0 & 0 \\
  0 & 0 & h_1(4,n,2) & h_1(5,n,2) & h_1(6,n,2) & h_1(7,n,2) & h_1(8,n,2)
 \end{pmatrix}\nonumber\\ &\equiv \begin{pmatrix}
  1 & 1 & 2 & 1 & 0 & 0 & 0 \\
  0 & 0 & 4 & 0 & 1 & 1 & 1
 \end{pmatrix}\pmod{5}\label{hmodALL}
\end{align}
\end{corollary}

Notice that this gives us the initial part of $\Omega$ as shown in (\ref{omegadefinition}).
\section{Main Theorem}\label{maintheoremsection}
We are now at the point that we may begin proving our main theorem.  Since we know that 
\begin{align*}
\frac{1}{5}L_1\in\mathcal{V}^{(1)}_6,
\end{align*} we want to show that applying $U^{(1)}$ to a member of $\mathcal{V}^{(1)}_n$ will produce a member of $\mathcal{V}^{(0)}_{n'}$ in which all coefficients are divisible by an extra power of 5.  Then applying $U^{(0)}$ to $\mathcal{V}^{(0)}_{n}$, one should achieve a member of $\mathcal{V}^{(1)}_{n'}$.
\subsection{First Intermediate Theorem}
However, this strategy will need to be slightly modified.  We start by proving the following slightly less ambitious theorem:
\begin{theorem}\label{v02v1}
Let $f\in\mathcal{V}_n^{(0)}$.  Then
\begin{align}
U^{(0)}\left( f \right)\in\hat{\mathcal{V}}_{5n+6}.
\end{align}
\end{theorem}
\begin{proof}
Let $f\in\mathcal{V}_n^{(0)}$.  Then we can express $f$ as
\begin{align*}
f = \frac{1}{(1+5x)^n}\sum_{m\ge 1} s(m)\cdot 5^{\theta_0(m)}\cdot x^m.
\end{align*}  
\begin{align}
U^{(0)}\left( f \right) &= \sum_{m\ge 1} s(m)\cdot 5^{\theta_0(m)}\cdot U^{(0)}\left( \frac{x^m}{(1+5x)^n}\right)\\
&=\frac{1}{(1+5x)^{5n+6}}\sum_{m\ge 1}\sum_{r\ge \left\lceil (m+1)/5 \right\rceil} s(m)\cdot h_0(m,n,r) 5^{\theta_0(m) + \pi_0(m,r)}\cdot x^r\\
&=\frac{1}{(1+5x)^{5n+6}}\sum_{r\ge 1}\sum_{m\ge 1} s(m)\cdot h_0(m,n,r) 5^{\theta_0(m) + \pi_0(m,r)}\cdot x^r.
\end{align}  We need to prove that for $m+1\le 5r$,
\begin{align}
\theta_0(m) + \pi_0(m,r) - \theta_1(r)\ge 0.\label{v02v1nonneg}
\end{align}  
\subsubsection{}$\left\lceil \frac{m+1}{5} \right\rceil \le r\le \left\lceil \frac{m+33}{5} \right\rceil-1$\\
\noindent We have $\pi_0(m,r)=0$.  Notice also that $m\ge 5r-28$.

For $1\le m\le 4$, we have $\theta_0(m)=0$.  We also have $1 \le r\le 7$, and $\theta_1(r)=0$.  Thus (\ref{v02v1nonneg}) follows trivially.

For $m\ge 5$ we have $\theta_0(m) = \left\lfloor \frac{5m-1}{7} \right\rfloor-2$, and $r\ge 2$.  If $2\le r\le 7$, then (\ref{v02v1nonneg}) again follows trivially.  If $r\ge 8$, then
\begin{align*}
\theta_0(m) + \pi_0(m,r) &= \left(\left\lfloor \frac{5m-1}{7} \right\rfloor - 2\right) + 0 - \left( \left\lfloor \frac{5r-2}{7} \right\rfloor - 5 \right)\\
&\ge \left\lfloor \frac{25m-29}{7} \right\rfloor - \left\lfloor \frac{5r-2}{7} \right\rfloor +3 \\
&\ge \left\lfloor \frac{5r-2}{7} \right\rfloor + \left\lfloor \frac{20r-1}{7} \right\rfloor - \left\lfloor \frac{5r-2}{7} \right\rfloor -1 \\
&\ge \left\lfloor \frac{20r-1}{7} \right\rfloor - 1\ge 0.
\end{align*}
\subsubsection{}$r\ge \left\lceil \frac{m+33}{5} \right\rceil$\\
\noindent We have $\pi_0(m,r)=\left\lfloor \frac{5r-m+2}{7} \right\rfloor - 5$.

For $1\le m\le 4$, we have $\theta_0(m)=0$ and $r\ge 7$.  For $r=7$, $\theta_1(r)=0$, and 
\begin{align*}
\theta_0(m) + \pi_0(m,r) &= \left\lfloor \frac{5r-m+2}{7} \right\rfloor - 5\ge \theta_1(r).
\end{align*}  For $r\ge 8$, we have
\begin{align*}
\theta_0(m) + \pi_0(m,r) &= \left\lfloor \frac{5r-m+2}{7} \right\rfloor - 5\\
&\ge \left\lfloor \frac{5r-2}{7} \right\rfloor - 5 = \theta_1(r).
\end{align*}  For $m\ge 5$,
\begin{align*}
\theta_0(m) + \pi_0(m,r) &= \left\lfloor \frac{5m-1}{7} \right\rfloor - 2 + \left\lfloor \frac{5r-m+2}{7} \right\rfloor - 5\\
&\ge \left\lfloor \frac{5r+4m-5}{7} \right\rfloor - 7\\
&\ge \left(\left\lfloor \frac{5r-2}{7} \right\rfloor - 5\right) + \left\lfloor \frac{4m-3}{7} \right\rfloor - 2\\
&\ge \theta_1(r).
\end{align*}
\end{proof}
\subsection{Second Intermediate Theorem}
So far, this is what one would more or less expect from the classical proofs of congruence families.  However, the matter becomes more complicated on application of $U^{(1)}$, as the following theorem shows:

\begin{theorem}\label{gofrom1to0}
Let $f\in\mathcal{V}_n^{(1)}$ with $n\equiv 1\pmod{5}$.  Then
\begin{align}
\frac{1}{5}U^{(1)}\left( f \right)&\in\mathcal{V}_{5n}^{(0)}.
\end{align}
\end{theorem}
\begin{proof}
Let us take a function $f$ of the form
\begin{align*}
f &= \frac{1}{(1+5x)^n}\sum_{m\ge 1} s(m)\cdot 5^{\theta_1(m)}\cdot x^m.
\end{align*}  Applying $U^{(1)}$, we have
\begin{align*}
U^{(1)}\left( f \right) &= \frac{1}{(1+5x)^{5n}}\sum_{m\ge 2}\sum_{r\ge \left\lceil m/5 \right\rceil} s(m)\cdot h_1(m,n,r)\cdot 5^{\theta_1(m) + \pi_1(m,r)}\cdot x^r\\
&= \frac{1}{(1+5x)^{5n}}\sum_{r\ge 1}\sum_{m\ge 2} s(m)\cdot h_1(m,n,r)\cdot 5^{\theta_1(m) + \pi_1(m,r)}\cdot x^r.
\end{align*}  For $m\le 5r$ we consider the relation
\begin{align}
\theta_1(m) + \pi_1(m,r) - \theta_0(r) - 1 \ge 0.\label{v12v0nonneg}
\end{align}  We first verify that (\ref{v12v0nonneg}) is true for $r\ge 3$.
\subsubsection{}$r=3$
Here $\theta_0(3)=0$, and $m$ ranges over $2\le m\le 15$.

For $2\le m\le 7$, $\theta_1(m)=0$ and
\begin{align*}
\theta_1(m) + \pi_1(m,3) &= \left\lfloor \frac{15-m}{7} \right\rfloor\\
&\ge \left\lfloor \frac{15-7}{7} \right\rfloor\\
&\ge 1 = \theta_0(3) + 1.
\end{align*}

For $m=8$,
\begin{align*}
\theta_1(8) + \pi_1(8,3) &= 2 + 1\ge 1 = \theta_0(3) + 1.
\end{align*}

For $m\ge 9$,
\begin{align*}
\theta_1(m) + \pi_1(m,3) &= \left( \left\lfloor \frac{5m-2}{7} \right\rfloor - 5 \right)+ \left\lfloor \frac{15-m}{7} \right\rfloor\\
&\ge \left\lfloor \frac{15+4m-8}{7} \right\rfloor - 5\\
&\ge 1 = \theta_0(3) + 1.
\end{align*}
\subsubsection{}$r=4$
Here $\theta_0(4)=0$, and $m$ ranges over $2\le m\le 20$.

For $2\le m\le 7$, $\theta_1(m)=0$ and
\begin{align*}
\theta_1(m) + \pi_1(m,4) &= \left\lfloor \frac{20-m}{7} \right\rfloor\\
&\ge \left\lfloor \frac{20-7}{7} \right\rfloor\\
&\ge 1 = \theta_0(4) + 1.
\end{align*}

For $m\ge 8$, 
\begin{align*}
\theta_1(m) + \pi_1(m,4) &= \left( \left\lfloor \frac{5m-2}{7} \right\rfloor - 5 \right) +\left\lfloor \frac{20-m}{7} \right\rfloor\\
&\ge \left\lfloor \frac{20+4m-8}{7} \right\rfloor - 5\\
&\ge 1 = \theta_0(4) + 1.
\end{align*}
\subsubsection{}$r\ge 5$
Here $\theta_0(r)=\left( \left\lfloor \frac{5r-1}{7} \right\rfloor - 2 \right)$, and $m\ge 2$.

For $2\le m\le 7$, $\theta_1(m)=0$ and
\begin{align*}
\theta_1(m) + \pi_1(m,r) &= \left\lfloor \frac{5r-7}{7} \right\rfloor\\
&\ge \left\lfloor \frac{5r}{7} \right\rfloor - 1\\
&\ge \left\lfloor \frac{5r-1}{7} \right\rfloor - 1\\
&=\theta_0(r)+1.
\end{align*}

For $m=8$, 
\begin{align*}
\theta_1(8) + \pi_1(8,r) &= 0 + \left\lfloor \frac{5r-8}{7} \right\rfloor\\
&= \left\lfloor \frac{5r-1}{7} \right\rfloor - 1\\
&= \theta_0(r) + 1.
\end{align*}

For $m\ge 9$, 
\begin{align*}
\theta_1(m) + \pi_1(m,r) &= \left( \left\lfloor \frac{5m-2}{7} \right\rfloor - 5 \right) + \left\lfloor \frac{5r-m}{7} \right\rfloor\\
&\ge \left\lfloor \frac{5r+4m-8}{7} \right\rfloor - 5\\
&\ge \left\lfloor \frac{5r-1}{7} \right\rfloor + \left\lfloor \frac{4m}{7} \right\rfloor - 6\\
&\ge \left\lfloor \frac{5r-1}{7} \right\rfloor - 1\\
&= \theta_0(r) + 1.
\end{align*}
\subsubsection{}$1\le r\le 2$
Let us first consider the case for $r=1$.  We know that we will receive a contribution for $2\le m\le 5$.  In each of these cases, we see that
\begin{align*}
\theta_1(m) + \pi_1(m,1) &= 0 + \left\lfloor \frac{5-m}{7} \right\rfloor=0\\
&\neq \theta_0(1)+1=1.
\end{align*}  On the other hand, for $r=2$, we have already shown that we get the correct contribution for $m\ge 9$.  Moreover, for $2\le m\le 3$,
\begin{align*}
\theta_1(m) + \pi_1(m,2) &= 0 + \left\lfloor \frac{10-m}{7} \right\rfloor\ge 1\\
&=\theta_0(2)+1.
\end{align*}  However, for $4\le m\le 7$,
\begin{align*}
\theta_1(m) + \pi_1(m,2) &= 0 + \left\lfloor \frac{10-m}{7} \right\rfloor=0\\
&\neq \theta_0(1)+1=1.
\end{align*}

The only possible resolution to this problem is that 
\begin{align*}
\sum_{m=2}^5 s(m)\cdot h_1(m,n,1)&\equiv 0\pmod{5},\\
\sum_{m=4}^8 s(m)\cdot h_1(m,n,2)&\equiv 0\pmod{5}.
\end{align*}  By Corollary \ref{congcor1b}, we can reduce this to
\begin{align}
s(2)+s(3)+2s(4)+s(5)&\equiv 0\pmod{5},\label{kernelmotivationA}\\
4s(4)+s(6)+s(7)+s(8)&\equiv 0\pmod{5}.\label{kernelmotivationB}
\end{align}  But notice that these are the components of the image of $\Omega$ applied to $(s(m))_{m\ge 2}$.  Because $(s(m))_{m\ge 2}\in\mathrm{ker}\left( \Omega \right)$, these relations must both reduce to 0, and we have accounted for the necessary powers of 5 in all cases.
\end{proof}
\subsection{Stability Theorem}
We now see the necessity of the kernel properties of $\mathcal{V}_n^{(1)}$.  However, we have not done enough.  We need to prove \textit{stability} of the properties of $\mathcal{V}_n^{(1)}$.  That is, we need to show the following:

\begin{theorem}\label{oddbacktoodd}
Let $f\in\mathcal{V}_n^{(1)}$, with $n\equiv 1\pmod{5}$.  Then
\begin{align}
\frac{1}{5}U^{(0)}\circ U^{(1)}\left( f \right)&\in\mathcal{V}_{25n+6}^{(1)}.\label{finalstabilitycheck}
\end{align}
\end{theorem}

It is this which will finally allow us to complete the proof of Theorem \ref{Mythm}, and with it Theorem \ref{Thm12}.
\begin{proof}
We suppose that $f$ has the form
\begin{align*}
f &= \frac{1}{(1+5x)^n}\sum_{m\ge 1} s(m)\cdot 5^{\theta_1(m)}\cdot x^m.
\end{align*}  We proved in Theorem \ref{gofrom1to0} above that
\begin{align*}
\frac{1}{5}U^{(1)}\left( f \right)&\in\mathcal{V}_{5n}^{(0)}.
\end{align*}   Moreover, we know from Theorem \ref{v02v1} that
\begin{align*}
\frac{1}{5}U^{(0)}\circ U^{(1)}\left( f \right)&\in\hat{\mathcal{V}}_{25n+6}.
\end{align*}  Because of this, we can expand
\begin{align*}
\frac{1}{5}U^{(0)}\circ U^{(1)}\left( f \right) &= \frac{1}{(1+5x)^{25n+6}}\sum_{w\ge 2}t(w)\cdot 5^{\theta_1(w)}\cdot x^w,
\end{align*} with $t(w)$ defined as
\begin{align*}
t(w) :=& \sum_{r=1}^{5w-6}\sum_{m=1}^{5r}s(m)\cdot h_1(m,n,r)\cdot h_0(r,5n,w)\\ &\times 5^{\theta_1(m) + \pi_1(m,r) + \pi_0(r,w) -\theta_1(w) -1}.
\end{align*}  All that remains is to prove that $\left( t(w) \right)_{w\ge 2}\in\mathrm{ker}\left( \Omega \right)$.

Notice that for any fixed $w$, the values $m,r$ are bounded.  The variable $n$ is unbounded.  However, we can now take advantage of Theorem \ref{congcor1a}, and the fact that $n\equiv 1\pmod{5}$ by hypothesis, and we define
\begin{align*}
\hat{t}(w) :=& \sum_{r=1}^{5w-6}\sum_{m=1}^{5r}s(m)\cdot h_1(m,1,r)\cdot h_0(r,5,w)\\ &\times 5^{\theta_1(m) + \pi_1(m,r) + \pi_0(r,w) -1},
\end{align*}  Obviously, $\hat{t}(w)\equiv t(w)\pmod{5}$.  If we examine $\hat{t}(w)$ for $2\le w\le 8$, we have
\begin{align*}
\hat{t}(2) =& \frac{1}{5}\left( -624 s(2) - 1664 s(3) + 94204 s(4) + 99616 s(5) + 57078 s(6) + 19008 s(7) + 3708 s(8) \right),\\
\hat{t}(3) =& \frac{1}{5}\left( 28224 s(2) + 75264 s(3) - 3621954 s(4) - 3834516 s(5) - 2197503 s(6) - 731808 s(7) - 142758 s(8) \right),\\
\hat{t}(4) =& 715008 s(2) + 1906688 s(3) - 67390288 s(4) - 71546272 s(5) - 41020056 s(6) - 13660416 s(7) - 2664816 s(8),\\
\hat{t}(5) =& 25337256 s(2) + 67566016 s(3) - 6656426 s(4) - 33820304 s(5) - 21775257 s(6) - 7251552 s(7) - 1414602 s(8),\\
\hat{t}(6) =& 457837968 s(2) + 1220901248 s(3) + 140880948572 s(4) + 147501656288 s(5) + 84383715654 s(6)\\ &+ 28101294144 s(7) + 5481881244 s(8),\\
\hat{t}(7) =& \frac{1}{5}( 18893919144 s(2) + 50383784384 s(3) + 38559332136626 s(4) + 40484108853704 s(5) + 23170599952857 s(6)\\ &+ 7716226285152 s(7) + 1505248688202 s(8) ),\\
\hat{t}(8) =& \frac{1}{5}( -116748977604 s(2) - 311330606944 s(3) + 1303629422734184 s(4) + 1369503027522686 s(5)\\ &+ 783890939008863 s(6) + 261049773444768 s(7) + 50924482319718 s(8) ).
\end{align*}  This computation is explicitly carried out in our Mathematica supplement.

It is clear that (so long as $s(m)$ is always integer-valued) $\hat{t}(4), \hat{t}(5), \hat{t}(6)\in \mathbb{Z}$.  For the remainder, we define the ideal 
\begin{align*}
I := \left< s(2)+s(3)+2s(4)+s(5), 4s(4)+s(6)+2s(7)+s(8) \right>\le \mathbb{Z}/5\mathbb{Z}[s(2),s(3),...,s(8)].
\end{align*}  By a simple polynomial reduction through a computer algebra system, it can be shown that
\begin{align*}
5\hat{t}(2), 5\hat{t}(3), 5\hat{t}(7), 5\hat{t}(8)\in I.
\end{align*}  This ensures that
\begin{align*}
\hat{t}(2), \hat{t}(3), \hat{t}(7), \hat{t}(8)\in \mathbb{Z}.
\end{align*}  We now need to verify that
\begin{align}
\hat{t}(2)+\hat{t}(3)+2\hat{t}(4)+\hat{t}(5)&\equiv 0\pmod{5},\label{firstreltgt0}\\
4\hat{t}(4)+\hat{t}(6)+\hat{t}(7)+\hat{t}(8)&\equiv 0\pmod{5}.\label{firstreltgt1}
\end{align}  Expanding (\ref{firstreltgt0}) and (\ref{firstreltgt1}), we find that
\begin{align}
\hat{t}(2)+\hat{t}(3)+2\hat{t}(4)+\hat{t}(5)=& 26772792 s(2) + 71394112 s(3) - 142142552 s(4)\\ &- 177659828 s(5) - 104243454 s(6) - 34714944 s(7) - 6772044 s(8),\nonumber\\
4\hat{t}(4)+\hat{t}(6)+\hat{t}(7)+\hat{t}(8)=& -19110313692 s(2) - 50960836512 s(3) + 268578362361582 s(4)\\ &+ 282144642746478 s(5) + 161496527427774 s(6) + 53781246598464 s(7)\nonumber\\ &+ 
 10491417423564 s(8).\nonumber
\end{align}  We reduce both of these polynomials modulo $I$ to achieve 0.

Because $t(w)\equiv \hat{t}(w)\pmod{5}$, we have $\left( t(w) \right)_{w\ge 2}\in\mathrm{ker}\left( \Omega \right)$, and we have therefore proved (\ref{finalstabilitycheck}).
\end{proof}
We at last have enough that we can complete our main theorem.
\begin{proof}[Proof of Theorem \ref{Mythm}]
From (\ref{isittruethough})-(\ref{isittruethoughb}), we know that
\begin{align}
L_1 = U^{(0)}(1).\label{aboutL1andu0}
\end{align}  We will prove this along with (\ref{L1defninx}) in Section \ref{initialrelationssection} below.

We already proved in (\ref{L1inducstage1b}) of Section \ref{algebrastructure} that 
\begin{align*}
\frac{1}{5}\cdot L_1 = f_1 \in\mathcal{V}^{(1)}_6.
\end{align*}  Let us now suppose that for some $\alpha\ge 1$, we have
\begin{align*}
\frac{1}{5^{\alpha}}\cdot L_{2\alpha-1}\in\mathcal{V}_n^{(1)}.
\end{align*}   Then there exists some $f_{2\alpha-1}\in\mathcal{V}_n^{(1)}$ such that
\begin{align}
L_{2\alpha-1} = 5^{\alpha}\cdot f_{2\alpha-1}.
\end{align}  On applying $U^{(1)}$, we have
\begin{align}
L_{2\alpha} = U_5\left(L_{2\alpha-1} \right) = U_5\left( 5^{\alpha}\cdot f_{2\alpha-1} \right) = 5^{\alpha}\cdot U^{(1)}\left( f_{2\alpha-1} \right).
\end{align}  But by Theorem \ref{gofrom1to0}, there exists some $f_{2\alpha}\in\mathcal{V}_{5n}^{(0)}$ such that
\begin{align}
U^{(1)}\left( f_{2\alpha-1} \right) = 5\cdot f_{2\alpha}\label{important101et}
\end{align}  Therefore, we have
\begin{align}
L_{2\alpha} = 5^{\alpha+1}\cdot f_{2\alpha}.
\end{align}  Moreover, on applying $U^{(0)}$ we have
\begin{align}
L_{2\alpha+1} = U_5\left( \mathcal{A}\cdot L_{2\alpha} \right) = U_5\left( \mathcal{A}\cdot 5^{\alpha+1}\cdot f_{2\alpha}\right) = 5^{\alpha+1}\cdot U^{(0)}\left( f_{2\alpha} \right).\label{prooffinalstep2ap1}
\end{align}  We recall from (\ref{important101et}) that
\begin{align}
U^{(0)}\left( f_{2\alpha} \right)=\frac{1}{5}U^{(0)}\circ U^{(1)}\left( f_{2\alpha-1} \right).
\end{align}  Because $f_{2\alpha-1}\in\mathcal{V}_{n}^{(1)}$, we know that by Theorem \ref{oddbacktoodd},
\begin{align}
U^{(0)}\left( f_{2\alpha} \right)\in\mathcal{V}_{25n+6}^{(1)}.
\end{align}  Therefore, there exists some $f_{2\alpha+1}\in\mathcal{V}_{25n+6}^{(1)}$ such that
\begin{align}
U^{(0)}\left( f_{2\alpha} \right) = f_{2\alpha+1}.
\end{align}  Inserting this into \ref{prooffinalstep2ap1}, we have
\begin{align}
L_{2\alpha+1} = 5^{\alpha+1}\cdot f_{2\alpha+1}.
\end{align}  By induction, we have
\begin{align*}
\frac{L_{2\alpha-1}}{5^{\alpha}}&\in\mathcal{V}_{n_{2\alpha-1}}^{(1)},\\
\frac{L_{2\alpha}}{5^{\alpha+1}}&\in\mathcal{V}_{n_{2\alpha}}^{(0)}.
\end{align*}  We need only verify that
\begin{align*}
n_{2\alpha-1} &=\psi(2\alpha-1),\\
n_{2\alpha} &=\psi(2\alpha)-1.
\end{align*}  We note that for all $\alpha\ge 1$,
\begin{align*}
5^{\alpha}&\equiv 1\pmod{4}.
\end{align*}  We can then write
\begin{align*}
\psi(\alpha)=\left\lfloor \frac{5^{\alpha+1}}{4} \right\rfloor &= \frac{5^{\alpha+1}}{4} - \frac{1}{4}.
\end{align*}  With this in mind we can quickly verify that
\begin{align*}
5\cdot\psi(\alpha) &= 5\cdot \left\lfloor \frac{5^{\alpha+1}}{4} \right\rfloor\\
&= 5\cdot \left(\frac{5^{\alpha+1}}{4} - \frac{1}{4}\right)\\
&= \frac{5^{\alpha+2}}{4} - \frac{1}{4} - 1\\
&= \psi(\alpha+1) - 1.
\end{align*}  We then have
\begin{align*}
5\cdot\psi(2\alpha-1) &= \psi(2\alpha) - 1,\\
5\cdot\left(\psi(2\alpha)-1\right)+6 &= 5\cdot\psi(2\alpha)-5+6 = \psi(2\alpha+1).
\end{align*}  Because applying $U^{(1)}$ to a rational polynomial with a denominator of $(1+5x)^n$ will result in a rational polynomial with denominator $(1+5x)^{5n}$, and similarly applying $U^{(0)}$ will result in a denominator $(1+5x)^{5n+6}$, we need only verify that the denominator of $L_1$ has power $\psi(1)=6$.  This is true from (\ref{L1defninx}).
\end{proof}
\section{Initial Relations}\label{initialrelationssection}
We now need to finish the proofs of Theorem \ref{thmuyox}, Corollary \ref{congcor1b}, and the modular equation (\ref{modX}).  We also want to prove (\ref{L1defninx}).

To finish the proof of Theorem \ref{thmuyox}, we need to verify the existence of $h_i(m,n,r)$ for $1\le m\le 5$, $1\le n\le 5$.  This gives us 25 initial relations; for $i=0,1$, we have 50 such relations.

However, these relations are algebraically dependent on a much smaller number of relations.  Notice that
\begin{align}
U^{(i)}\left( \frac{x^{m}}{(1+5x)^{n}} \right) =& \frac{1}{5^m}\cdot U^{(i)}\left( \frac{(z-1)^{m}}{z^{n}} \right)\\
=& \frac{1}{5^m}\sum_{r=0}^{m}(-1)^{m-r}{{m}\choose{r}}\cdot U^{(i)}\left( z^{r-n} \right)\\
=& \frac{1}{5^m}\sum_{r=0}^{m}(-1)^{m-r}{{m}\choose{r}}\cdot U^{(i)}\left( (1+5x)^{r-n} \right).
\end{align}  Therefore if we understand $U^{(i)}(z^n)$ for all $n$, then we can compute $U^{(i)}\left(x^m/(1+5x)^n\right)$.  Of course,
\begin{align}
U^{(i)}\left( (1+5z)^n \right) &= -\sum_{k=0}^{n} b_k(\tau)\cdot U^{(i)}\left( (1+5x)^{k+n-5} \right).
\end{align}  If we consider $U^{(i)}\left( (1+5z)^n \right)$ for positive $n$, then we have
\begin{align}
U^{(i)}\left( (1+5x)^n \right) &= \sum_{k=0}^n {{n}\choose{k}}\cdot 5^k\cdot U^{(i)}\left( x^k \right).
\end{align}  So to completely construct $U^{(i)}\left(x^m/(1+5x)^n\right)$, we need to know how to construct $U^{(i)}(x^n)$ for all integers $n$.  Using (\ref{modX}), we only need five initial relations, e.g., the relations for $0\le n\le 4$.  Again accounting for two different values of $i$, this gives us a total of ten initial relations.  We give the explicit relations in the Appendix below.

Notice that by (\ref{aboutL1andu0}), we can prove (\ref{L1defninx}) by deriving the relation for $U^{(0)}(1)$.  This is one of our ten initial relations.

From these relations, one can work algorithmically to the construction of $U^{(i)}\left(x^m/(1+5x)^n\right)$ for any $m,n\in\mathbb{Z}$.  We provide such a construction in our Mathematica supplement, which can be found online at \url{https://www.risc.jku.at/people/nsmoot/d5congsuppA.nb}.  We emphasize that our construction is by no means the most efficient.  Rather, we have sacrificed efficiency for an easily comprehensible construction process.  Nevertheless, the 50 initial cases for Theorem \ref{thmuyox} can still be constructed almost immediately.

Once these relations are constructed, we can quickly compute $h_i(m,n,r)$.  Due to Theorem \ref{congcor1a}, we can verify Corollary \ref{congcor1b} by checking the congruence properties of $h_i(m,n,r)$, with the appropriate $m,r$, for $1\le n\le 5$.  This is also computed in our Mathematica supplement.
\subsection{Computing the Fundamental Relations}
We need to verify the relations (\ref{initial10})-(\ref{initial04}) in the Appendix.  Notice that these relations have the form
\begin{align}
U^{(1)}\left( x^l \right) &=p_{1,l}(x)\in\mathbb{Z}[x],\\
z^6\cdot U^{(0)}\left( x^l \right) &=p_{0,l}(x)\in\mathbb{Z}[x],
\end{align} for $0\le l\le 4$.

On the other hand, $x,z$ have poles at the cusp $[0]_{10}$.  If we divide both sides of each relation by the highest power of $x$ appearing on the right-hand side, then we have 
\begin{align}
\frac{1}{x^{m}}\cdot U^{(1)}\left( x^l \right) &\in\mathbb{Z}[x^{-1}]\subseteq\mathcal{M}^{\infty}(\Gamma_0(10)),\label{initialTypeA}\\
\frac{1}{x^{m}}\cdot z^6\cdot U^{(0)}\left( x^l \right) &\in\mathbb{Z}[x^{-1}]\subseteq\mathcal{M}^{\infty}(\Gamma_0(10)).\label{initialTypeB}
\end{align}  This would more convenient to check, given that it is much easier to compute the principal part of a modular function with a pole only at $[\infty]_{10}$.  Using Ligozat's theorem, we can compute the order of $\mathcal{A}(\tau)$, $x(\tau)$, $x(5\tau)$, $z(5\tau)$ as functions on the congruence subgroup $\Gamma_0(50)$.
\begin{table}[hbt!]
\begin{center}
\begin{tabular}{l|c|c|c|r}
 Cusp Representative      &$\mathcal{A}(\tau)$  & $x(\tau)$ & $x(5\tau)$ & $z(5\tau)$\\
\hline
 $\infty$      & 6          & 1 & 5 & 0  \\
 $1/25$         & 27          & 0 & 0 & 0  \\
 $1/10$         & 0          & 1 & 0 & 1  \\
 $1/5$         & 0          & 0 & -1& -1 \\
 $3/10$         & 0         & 1 & 0 & 1  \\
 $2/5$         & 0          & 0 & -1 & -1  \\
$1/2$          & -6          & 0 & 0 & 1  \\
 $3/5$         & 0          & 0 & -1 & -1  \\
 $7/10$         & 0          & 1 & 0 & 1  \\
 $4/5$         & 0          & 0 & -1 & -1  \\
 $9/10$         & 0          & 1 & 0 & 1  \\
$0$             & -27         & -5 & -1 & -1  
 \end{tabular}
\caption{Order at Cusps of $\mathrm{X}_0(50)$}\label{tablew0}
\end{center}
\end{table}

We can now pull the factors $1/x^m$ and $z^6$ inside the $U^{(i)}$ operators, so that we have
\begin{align*}
\frac{1}{x^{m}}\cdot U^{(1)}\left( x^l \right) &= U_5\left( x(5\tau)^{-m} x(\tau)^l \right),\\
\frac{1}{x^{m}}\cdot z^6\cdot U^{(0)}\left( x^l \right) &= U_5\left( \mathcal{A}(\tau) x(5\tau)^{-m} z(5\tau)^6 x^l \right).
\end{align*}  By Table \ref{tablew0}, we can verify that
\begin{align*}
x(5\tau)^{-5} x(\tau)^l&\in\mathcal{M}^{\infty}(\Gamma_0(50)),\\
\mathcal{A}(\tau) x(5\tau)^{-5} z(5\tau)^6 x(\tau)^l&\in\mathcal{M}^{\infty}\left(\Gamma_0(50)\right).
\end{align*}
\begin{lemma}
If $f\in\mathcal{M}^{\infty}(\Gamma_0(50))$, then $U_5\left(f\right)\in\mathcal{M}^{\infty}(\Gamma_0(10))$.
\end{lemma}
\begin{proof}
We can express $U_5\left(f\right)$ as
\begin{align*}
5\cdot U_5\left(f(\tau)\right) = \sum_{r=0}^4 f\left( \frac{\tau+r}{5}\right).
\end{align*}  As such, $U_5\left(f\right)$ will have a pole wherever $f\left( \frac{\tau+r}{5}\right)$ has a pole for $0\le r\le 4$.

Let us suppose that $\tau$ approaches the rational point $h/k\in\mathbb{Q}$, with $\mathrm{gcd}(h,k)=1$.  Then
\begin{align*}
\frac{\tau+r}{5}&\rightarrow \frac{h+kr}{5k}.
\end{align*}  If $f$ has a pole at $\frac{h+kr}{5k}$, then $\frac{h+kr}{5k}\in[\frac{1}{50}]_{50}$.  Thus, by Lemma \ref{cuspEqualLemma}, there exist integers $j,y$ such that
\begin{align}
&y\equiv h+kr+5jk\pmod{50},\label{cuspto50a}\\
&50\equiv 5ky\pmod{50},\label{cuspto50b}\\
&\mathrm{gcd}(y,50)=1.\label{cuspto50c}
\end{align}  By (\ref{cuspto50b}), (\ref{cuspto50b}), we must have $k\equiv 0\pmod{10}$, and therefore $5k\equiv 0\pmod{50}$.  We therefore set $k=10m$ for $m\in\mathbb{Z}$.

We want to show that $\frac{h+kr}{50m}\in [\frac{1}{10}]_{10}$.  Again using Lemma \ref{cuspEqualLemma}, we need to demonstrate that there exist $y,j\in\mathbb{Z}$ such that
\begin{align*}
&y\equiv h+kr+50mj\pmod{10},\\
&10\equiv 50my\pmod{10},\\
&\mathrm{gcd}(y,10)=1.
\end{align*}  Simplifying, we have
\begin{align*}
&y\equiv h+kr\pmod{10},\\
&\mathrm{gcd}(y,10)=1.
\end{align*}  Notice that $\mathrm{gcd}(h+kr,10)=\mathrm{gcd}(h+50m,10)=\mathrm{gcd}(h,10)=1$, since $\mathrm{gcd}(h,k)=\mathrm{gcd}(h,50m)=1$. Therefore, we can simply set $y=h+kr$ and $j=0$.  Therefore, $\frac{h+kr}{50m}\in [\frac{1}{10}]_{10}$.  In other words, if $f$ has a pole only at the cusp $[\infty]_{50}$, then $U_5\left(f\right)$ has a pole only at $[\infty]_{10}$.
\end{proof}
We show in our supplement that 
\begin{align}
x(5\tau)^{-20} x(\tau)^l &\in\mathcal{M}^{\infty}(\Gamma_0(50)),\label{initialrel1riggedA}\\
\mathcal{A}(\tau) x(5\tau)^{-53} z(5\tau)^6 x(\tau)^l &\in\mathcal{M}^{\infty}(\Gamma_0(50)).\label{initialrel0riggedA}
\end{align}  It is now only a matter of applying the $U_5$ operator and computing these principal parts, and confirming that they match the right-hand sides of (\ref{initial10})-(\ref{initial04}), adjusted by multiplying $z^{6\cdot (1-i)}/x^m$, where $m$ is the degree of the right-hand side in $x$; for $i=0$ we take $m=53$, and for $i=1$ we take $m=20$, since these are the maximal degrees in their respective cases.  A more careful computation can of course be more efficient, but we are interested in simplicity of presentation for the time being.
\begin{corollary}
The witness identity (\ref{L1defninx}) is true.
\end{corollary}
\begin{proof}
We established in (\ref{aboutL1andu0}) that $U^{(0)}(1)=L_1$.  Taking (\ref{initialrel0riggedA}) for $l=0$, we therefore have
\begin{align}
\mathcal{A}(\tau) x(5\tau)^{-53} z(5\tau)^6 &\in\mathcal{M}^{\infty}(\Gamma_0(50)),\label{initialrel0riggedB}\\
U_5\left(\mathcal{A}(\tau) x(5\tau)^{-53} z(5\tau)^6\right) &\in\mathcal{M}^{\infty}(\Gamma_0(10)).\label{initialrel0riggedC}
\end{align}  In our supplement we compare with the right-hand side of (\ref{u01l1etc}) below, after multiplying by $z^6$ and dividing by $x^{53}$.
\end{proof}
\subsection{Proof of the Modular Equation}
As a final application of the modular cusp analysis, we can confirm using Table \ref{tablew0} that
\begin{align*}
x(5\tau)^{-5}\cdot x(\tau)\in\mathcal{M}^{\infty}\left(\Gamma_0(50)\right).
\end{align*}  Therefore, we can multiply $x(5\tau)^{-5}$ onto the left-hand side of (\ref{modX}) to show
\begin{align}
x(5\tau)^{-5}\cdot\left(x^5+\sum_{j=0}^4 a_j(5\tau) x^j\right)\in\mathcal{M}^{\infty}\left(\Gamma_0(50)\right).
\end{align}  As we show in our Mathematica supplement, it can be quickly computed that this modular function has no principal part, and constant term 0, confirming that it must be equal to 0.  This verifies (\ref{modX}).
\section{Discussion}\label{finalsection}
It is remarkable to note the implications of the importance of $\Omega$ in our proof of Theorem \ref{Thm12}.  In particular we are enormously interested in the importance of the kernel stability condition proved in Theorem \ref{oddbacktoodd}.  Not only is the vector $\left( s(m) \right)_{m\ge 2}$ associated with a given $f\in\mathcal{V}_n^{(1)}$ in the kernel of $\Omega$, but this property is \textit{genetic}: it is inherited by $\frac{1}{5}U^{(0)}\circ U^{(1)}\left( f \right)$.

It is this property which appears to explain why the congruence family is there at all.  Suppose we take a typical sequence of functions over, say, $\mathrm{X}_0(10)$, with integer-valued Fourier coefficients, in which each function is related to its successor via some well-defined modified $U_5$ and $T_5$ operators.  In this case, we might expect that this sequence would eventually become divisible by arbitrarily large powers of 5 (considering classical results, e.g, \cite[Theorem 8]{Serre}).  However, this is not universal if we restrict ourselves to, say, an alternating pair of $U_5$ operators as we have done here (see the discussions in \cite[Section 1.3]{Paule} and \cite[Section 6]{RaduS}).

And even if 5-adic convergence to 0 is ensured, it may be very slow and irregular.  This is because the members of our sequence will be composed of various different basis functions which will not all converge at the same rate.  We certainly do not expect the precise, clean pattern of $5$-adic convergence that we see for strict congruence families.  Sometimes a function in the sequence will be composed of basis functions in just the right manner that any anomalous terms not divisible by 5 will cancel out.  Nevertheless, successive functions may not have this property, in which case anomalous terms may occur later in the sequence.

However, very occasionally a sequence will exist in which the initial functions are composed of basis functions in so delicate and precise a manner that the cancellation condition is genetic, in which case the accumulation of anomalous terms is prevented at every step.  A congruence family is the manifestation of such a function sequence.

We see that the kernel of the congruence mapping $\Omega$ describes this cancellation condition; and stability of membership in the kernel describes the conditions in which a congruence family can be found.

Thus, some combinatorial objects (in our case, $d_5(n)$) exhibit a given precise set of divisibility properties as a result of the way that certain associated modular functions relate to one another.  In some ways, it is almost an inverse to the notion of a partition crank---that is, it gives an analytic explanation for why certain combinatorial properties exist.

The reason that this has been largely undiscovered before now is that for most classical congruence families (e.g., those for $p(n)$), the operator analogous to $\Omega$ is trivial: it sends every associated modular vector $\left(s(m)\right)_{m\ge 0}$ to zero, and the kernel is therefore the entire set.  This in some manner constitutes a measure of the relative simplicity of the associated spaces of modular functions (e.g., the curve $\mathrm{X}_0(5)$ associated with Ramanujan's congruences for $p(n)$ modulo powers of 5).  In the case of other families, as in Theorem \ref{Thm12}, the analogous operator (in our case, $\Omega$ itself) is nontrivial, and must be taken into account before a proper understanding of the congruence family can be achieved.

The most important question facing us appears to be \textit{under what conditions} the kernel of the associated congruence operator is trivial.  We believe that this is an especially promising long-term research question.
\pagebreak
\section{Appendix I}
\begin{align}
U^{(1)}\left( 1 \right) &=1\label{initial10}\\
U^{(1)}\left( x \right) &=41 x + 860 x^2 + 6800 x^3 + 24000 x^4 + 32000 x^5\\
U^{(1)}\left( x^2 \right) &=86 x + 10195 x^2 + 366600 x^3 + 6534800 x^4 + 68384000 x^5 + 
 450720000 x^6 + 1907200000 x^7\\& + 5056000000 x^8 + 7680000000 x^9 + 
 5120000000 x^{10}\nonumber\\
U^{(1)}\left( x^3 \right) &=51 x + 27495 x^2 + 2836265 x^3 + 128688900 x^4 + 3343692000 x^5 + 
 56283680000 x^6\\&\nonumber + 656205600000 x^7 + 5502096000000 x^8 + 
 33821312000000 x^9 + 153192960000000 x^{10}\\&\nonumber + 506956800000000 x^{11} + 
 1195008000000000 x^{12} + 1904640000000000 x^{13} + 
 1843200000000000 x^{14}\\&\nonumber + 819200000000000 x^{15}\nonumber\\
U^{(1)}\left( x^4 \right) &=12 x + 32674 x^2 + 8579260 x^3 + 831492275 x^4 + 42958434000 x^5 + 
 1396773180000 x^6\\& + 31314949600000 x^7 + 511802288800000 x^8 + 
 6319880448000000 x^9 + 60349364480000000 x^{10}\nonumber\\& + 
 452174745600000000 x^{11} + 2679038592000000000 x^{12} + 
 12574269440000000000 x^{13}\nonumber\\& + 46561935360000000000 x^{14} + 
 134544588800000000000 x^{15} + 297365504000000000000 x^{16}\nonumber\\& + 
 485949440000000000000 x^{17} + 553779200000000000000 x^{18} + 
 393216000000000000000 x^{19}\nonumber\\& + 131072000000000000000 x^{20}\nonumber\\
U^{(0)}\left( 1 \right) &= \frac{1}{(1+5x)^6}(5705x^2+6840120x^3+2034152125x^4+280484938650x^{5}+22921365211325x^{6}\label{u01l1etc}\\&+1260917405154520x^{7}
+50400843190048480x^{8}+1539115922208139200x^{9}\nonumber\\&+37183654303328448000x^{10}+728924483359472640000x^{11}+11816089262411136000000x^{12}\nonumber\\&+160681440628058880000000x^{13}+1853291134193264640000000x^{14}\nonumber\\
&+18284160727362809856000000x^{15}+155286793010086625280000000x^{16}\nonumber\\&+1140657222505472000000000000x^{17}+7269894420215070720000000000x^{18}\nonumber\\&+40277647277404979200000000000x^{19}
+194099187864646451200000000000x^{20}\nonumber\\&+813054581193729638400000000000x^{21} +2954545150241538048000000000000x^{22}\nonumber\\&+9282005730758492160000000000000x^{23} +25080951875200614400000000000000x^{24}\nonumber\\&+57872525958316032000000000000000x^{25}
+112916020309524480000000000000000x^{26}\nonumber\\&+183812885074411520000000000000000x^{27} +245082228994867200000000000000000x^{28}\nonumber\\&+260725452832768000000000000000000x^{29} +212837104353280000000000000000000x^{30}\nonumber\\&+125198296678400000000000000000000x^{31}
+47244640256000000000000000000000x^{32}\nonumber\\&+8589934592000000000000000000000x^{33})\nonumber.
\end{align}
\begin{align}
U^{(0)}\left( x \right) &=\frac{1}{(1+5x)^6}\big(1596 x^2 + 5311629 x^3 + 3020673965 x^4 + 693946917880 x^5 + 
 88012336687140 x^6\\& + 7203973630079449 x^7 + 417078090095103516 x^8 + 
 18123681491802321200 x^9\nonumber\\& + 615843754566814808000 x^{10} + 
 16857245810889643680000 x^{11} + 380061235251469769600000 x^{12}\nonumber\\& + 
 7179360858330987609600000 x^{13} + 115153019537919900211200000 x^{14}\nonumber\\& + 
 1584889574408762616832000000 x^{15} + 
 18875287036126384148480000000 x^{16}\nonumber\\& + 
 195815273618900539392000000000 x^{17} + 
 1778815480050553692160000000000 x^{18}\nonumber\\& + 
 14206927272980568637440000000000 x^{19} + 
 100058873107538207703040000000000 x^{20}\nonumber\\& + 
 622718357721614503116800000000000 x^{21} + 
 3428656886761288105984000000000000 x^{22}\nonumber\\& + 
 16707165479275661885440000000000000 x^{23} + 
 72013094304097396326400000000000000 x^{24}\nonumber\\& + 
 274190033219424878592000000000000000 x^{25} + 
 920048921836076924928000000000000000 x^{26}\nonumber\\& + 
 2711506126769477386240000000000000000 x^{27} + 
 6985969318446812364800000000000000000 x^{28}\nonumber\\& + 
 15637700398221885440000000000000000000 x^{29}\nonumber\\& + 
 30166681246666588160000000000000000000 x^{30}\nonumber\\& + 
 49621782059863244800000000000000000000 x^{31}\nonumber\\& + 
 68625507039156633600000000000000000000 x^{32}\nonumber\\& + 
 78284718944026624000000000000000000000 x^{33}\nonumber\\& + 
 71718179952394240000000000000000000000 x^{34}\nonumber\\& + 
 50720986785382400000000000000000000000 x^{35}\nonumber\\& + 
 25993142075392000000000000000000000000 x^{36}\nonumber\\& + 
 8589934592000000000000000000000000000 x^{37}\nonumber\\& + 
 1374389534720000000000000000000000000 x^{38}\big)\nonumber
\end{align}
\begin{align}
U^{(0)}\left( x^2 \right) &=\frac{1}{(1+5x)^6}\big(268 x^2 + 2847432 x^3 + 3155658820 x^4 + 1202043333790 x^5 + 
 233251365647870 x^6\\& + 27857600543181592 x^7 + 
 2282412359335489853 x^8 + 137474599581860685500 x^9\nonumber\\& + 
 6382595114107468178000 x^{10} + 236333942045815117200000 x^{11}\nonumber\\& + 
 7159344666536530274720000 x^{12} + 180943148092126406540160000 x^{13}\nonumber\\& + 
 3874442157232846507737600000 x^{14} + 
 71154385951639684561408000000 x^{15}\nonumber\\& + 
 1131931871787660010234880000000 x^{16} + 
 15724017231814160548864000000000 x^{17}\nonumber\\& + 
 191993970181152296671232000000000 x^{18} + 
 2071679521214318766489600000000000 x^{19}\nonumber\\& + 
 19840725410941052260024320000000000 x^{20} + 
 169240207863514163393331200000000000 x^{21}\nonumber\\& + 
 1289257956753574377619456000000000000 x^{22} + 
 8789154839425847410032640000000000000 x^{23}\nonumber\\& + 
 53694599262146012197683200000000000000 x^{24}\nonumber\\& + 
 294194062113759948701696000000000000000 x^{25}\nonumber\\& + 
 1445922789195393074724864000000000000000 x^{26}\nonumber\\& + 
 6372283630042655710248960000000000000000 x^{27}\nonumber\\& + 
 25156363490877335666688000000000000000000 x^{28}\nonumber\\& + 
 88814063288837761662976000000000000000000 x^{29}\nonumber\\& + 
 279743502586000693002240000000000000000000 x^{30}\nonumber\\& + 
 783588083068217747046400000000000000000000 x^{31}\nonumber\\& + 
 1943797721585087728844800000000000000000000 x^{32}\nonumber\\& + 
 4247467234308952948736000000000000000000000 x^{33}\nonumber\\& + 
 8120384728408910725120000000000000000000000 x^{34}\nonumber\\& + 
 13465927624462381875200000000000000000000000 x^{35}\nonumber\\& + 
 19155038490681409536000000000000000000000000 x^{36}\nonumber\\& + 
 23036102962346721280000000000000000000000000 x^{37}\nonumber\\& + 
 22969782997939650560000000000000000000000000 x^{38}\nonumber\\& + 
 18482412505792512000000000000000000000000000 x^{39}\nonumber\\& + 
 11532502585835520000000000000000000000000000 x^{40}\nonumber\\& + 
 5236424127283200000000000000000000000000000 x^{41}\nonumber\\& + 
 1539316278886400000000000000000000000000000 x^{42}\nonumber\\& + 
 219902325555200000000000000000000000000000 x^{43}\big)\nonumber
\end{align}
\begin{align}
U^{(0)}\left( x^3 \right) &=\frac{1}{(1+5x)^6}\big(25 x^2 + 1083750 x^3 + 2419268600 x^4 + 1533044850875 x^5 + 
 451892277223875 x^6\\& + 77776397020017600 x^7 + 8876969029993551625 x^8 + 727751880723215938525 x^9\nonumber\\& + 
 45237746915792486076500 x^{10} + 2216202089061921720156000 x^{11}\nonumber\\& + 88063763926314004467152000 x^{12} + 
 2901622367042821273526240000 x^{13}\nonumber\\& + 
 80660306943461135362236800000 x^{14} + 
 1918093917299266390588800000000 x^{15}\nonumber\\& + 
 39459730054782721266716160000000 x^{16} + 
 708795736244147980459443200000000 x^{17}\nonumber\\& + 
 11201811678733852133717299200000000 x^{18} + 
 156752971920695775627182080000000000 x^{19}\nonumber\\& + 
 1952565980789217283863347200000000000 x^{20}\nonumber\\& + 
 21745861234519182941633740800000000000 x^{21}\nonumber\\& + 
 217330181296111203156344832000000000000 x^{22}\nonumber\\& + 
 1954986367634769250929868800000000000000 x^{23}\nonumber\\& + 
 15867347909658771160720998400000000000000 x^{24}\nonumber\\& + 
 116421209420849214282399744000000000000000 x^{25}\nonumber\\& + 
 773301464019936672352829440000000000000000 x^{26}\nonumber\\& + 
 4654578055099292525988413440000000000000000 x^{27}\nonumber\\& + 
 25401759242372849348693196800000000000000000 x^{28}\nonumber\\& + 
 125704489323417659460550656000000000000000000 x^{29}\nonumber\\& + 
 563904329114669957135728640000000000000000000 x^{30}\nonumber\\& + 
 2291384934896565019580825600000000000000000000 x^{31}\nonumber\\& + 
 8423518909348688068345856000000000000000000000 x^{32}\nonumber\\& + 
 27966210361515509645574144000000000000000000000 x^{33}\nonumber\\& + 
 83658769333926385994956800000000000000000000000 x^{34}\nonumber\\& + 
 224821089074620615622656000000000000000000000000 x^{35}\nonumber\\& + 
 540741258958646666067968000000000000000000000000 x^{36}\nonumber\\& + 
 1158650329533680577413120000000000000000000000000 x^{37}\nonumber\\& + 
 2198994699321618726912000000000000000000000000000 x^{38}\nonumber\\& + 
 3670226839781351882752000000000000000000000000000 x^{39}\nonumber\\& + 
 5338951094289689477120000000000000000000000000000 x^{40}\nonumber\\& + 
 6691954927703970283520000000000000000000000000000 x^{41}\nonumber\\& + 
 7121217229055422627840000000000000000000000000000 x^{42}\nonumber\\& + 
 6307992271770668236800000000000000000000000000000 x^{43}\nonumber\\& + 
 4525788871530643456000000000000000000000000000000 x^{44}\nonumber\\& + 
 2526853642489692160000000000000000000000000000000 x^{45}\nonumber\\& + 
 1030022492900556800000000000000000000000000000000 x^{46}\nonumber\\& + 
 272678883688448000000000000000000000000000000000 x^{47}\nonumber\\& + 
 35184372088832000000000000000000000000000000000 x^{48}\big)\nonumber
\end{align}
\begin{align}
U^{(0)}\left( x^4 \right) &=\frac{1}{(1+5x)^6}\big(x^2 + 296324 x^3 + 1400331440 x^4 + 1491537289180 x^5 + 
 666654357758190 x^6\label{initial04}\\ 
&+ 164127484896644144 x^7 + 25818034832153226421 x^8 + 2843942677492333687050 x^9\nonumber\\ 
&+233247182152327524438975 x^{10} + 14876273050366789151536400 x^{11}\nonumber\\ &+ 761935276117868629545420000 x^{12} + 
 32118457754484194871176800000 x^{13}\nonumber\\ &+ 
 1135956912571962510189130400000 x^{14} + 
 34231269539693917027064800000000 x^{15}\nonumber\\ & + 
 889959895359737292913435520000000 x^{16} + 
 20168696740441852125087521280000000 x^{17}\nonumber\\ & + 
 401853770227841145701015936000000000 x^{18} + 
 7090293594621615513260024832000000000 x^{19}\nonumber\\ & + 
 111455620826575249212238069760000000000 x^{20}\nonumber\\ & + 
 1568967903706184088975184281600000000000 x^{21}\nonumber\\ & + 
 19865297696614393785791188992000000000000 x^{22}\nonumber\\ & + 
 227065713010913459721992798208000000000000 x^{23}\nonumber\\ & + 
 2350395554799549839893043609600000000000000 x^{24}\nonumber\\ & + 
 22090326137698781193612296192000000000000000 x^{25}\nonumber\\ & + 
 188920535524484098164060192768000000000000000 x^{26}\nonumber\\ & + 
 1472786798447713704939919769600000000000000000 x^{27}\nonumber\\ & + 
 10480689728875682601093981798400000000000000000 x^{28}\nonumber\\ & + 
 68152666405070998605987315712000000000000000000 x^{29}\nonumber\\ & + 
 405254490887530431856613785600000000000000000000 x^{30}\nonumber\\ & + 
 2204414239119153895179708006400000000000000000000 x^{31}\nonumber\\ & + 
 10970036975349264582008478105600000000000000000000 x^{32}\nonumber\\ & + 
 49929921460604339849107865600000000000000000000000 x^{33}\nonumber\\ & + 
 207727909385809764269938442240000000000000000000000 x^{34}\nonumber\\ & + 
 789224428145826945154469068800000000000000000000000 x^{35}\nonumber\\ & + 
 2734666462414638550878257152000000000000000000000000 x^{36}\nonumber\\ & + 
 8626866866620198090933534720000000000000000000000000 x^{37}\nonumber\\ & + 
 24722506007012126262789406720000000000000000000000000 x^{38}\nonumber\\ & + 
 64185847480498768484237312000000000000000000000000000 x^{39}\nonumber\\ & + 
 150467914734947763855818752000000000000000000000000000 x^{40}\nonumber\\ & + 
 317205685483135843705028608000000000000000000000000000 x^{41}\nonumber\\ & + 
 598387657370971073308262400000000000000000000000000000 x^{42}\nonumber\\ & + 
 1004032323013383412514816000000000000000000000000000000 x^{43}\nonumber\\ & + 
 1487355820084245534081024000000000000000000000000000000 x^{44}\nonumber\\ & + 
 1927442844590278705152000000000000000000000000000000000 x^{45}\nonumber\\ & + 
 2159710523624808618393600000000000000000000000000000000 x^{46}\nonumber\\ & + 
 2061307549285828316364800000000000000000000000000000000 x^{47}\nonumber\\ & + 
 1642771482986634280960000000000000000000000000000000000 x^{48}\nonumber\\ & + 
 1063549153298423480320000000000000000000000000000000000 x^{49}\nonumber
\end{align}
\begin{align*}& + 
 537321656791806771200000000000000000000000000000000000 x^{50}\\ & + 
 198721333557723136000000000000000000000000000000000000 x^{51}\\ & + 
 47850746040811520000000000000000000000000000000000000 x^{52}\\ & + 
 5629499534213120000000000000000000000000000000000000 x^{53}\big)
\end{align*}
\section{Acknowledgments}
The first author was funded by the Austrian Science Fund (FWF): W1214-N15, Project DK6.  The second author was funded in whole by the Austrian Science Fund (FWF): Einzelprojekte P 33933, ``Partition Congruences by the Localization Method".  Our most profound thanks to the Austrian Government and People for their generous support.

A huge thanks to the anonymous referee for their careful review and their constructive criticism which has substantially improved the paper.

\end{document}